\documentclass[10pt, reqno]{amsart}
\usepackage{esint}
\usepackage{slashed}
\usepackage{graphicx}
\usepackage{amsfonts}
\usepackage{amssymb}
\usepackage{amsmath}
\usepackage{amsxtra}
\usepackage{latexsym}
\usepackage{tensor}
\usepackage{enumitem}

\newtheorem{theorem}{Theorem}[section]
\newtheorem{lemma}[theorem]{Lemma}

\newtheorem{proposition}[theorem]{Proposition}
\newtheorem{definition}[theorem]{Definition}

\theoremstyle{remark}

\numberwithin{equation}{section}

\def\fB{\mathfrak{B}}
\def\fBb{\underline{\fB}}

\def\T{\mathcal{T}}

\def\bb{{\mathbf{b}}}

\def\Er{\mbox{Er}}

\def\sn{{\slashed{\nabla}}}

\def\zb{{\underline{\zeta}}}

\def\J{{\mathcal{J}}}

\def\M{{\mathcal{M}}}
\def\bT{{\textbf{T}}}

\def\bR{{\textbf{R}}}

\def\bd{{\textbf{D}}}
\def\ti{\tilde}

\def\bg{\mathbf{g}}

\def\hk{{\hat{k}}}
\def\I{{\mathcal I}}

\def\beaa{\begin{eqnarray*}}
\def\eeaa{\end{eqnarray*}}

\def\ba{\begin{array}}
\def\ea{\end{array}}
\def\d{\delta}
\def\be#1{\begin{equation} \label{#1}}
\def \eeq{\end{equation}}

\newcommand{\nn}{\nonumber}

\def\l{\langle}
\def\r{\rangle}

\def\nn{\nonumber}
\def\S{{\mathcal S}}

\def\S2{{\Bbb S}^2}

\def\Lb{\underline{L}}

\def\tr{\mbox{tr}}

\def\H{{\mathcal H}}

\def\N{{\mathcal N}}

\def\c{\cdot}

\def\a{\alpha}
\def\b{\beta}

\def\l{\langle}
\def\r{\rangle}
\def\ga{\gamma}
\def\Ga{\Gamma}

\def\p{\partial}

\def\nab{\nabla}

\def\Lb{{\underline{L}}}
\def\aaa{{\mathbf a}}

\def\tr{\mbox{tr}}
\def\Tr{\mbox{Tr}}

\def\tir{{\tilde r}}

\def\f14{\frac{1}{4}}
\def\f12{{\frac{1}{2}}}

\def\t1a{t^{-\frac{1}{a}}}

\def\bm{{\bf m}}

\def\sl{\slashed}
\def\sD{\slashed{\Delta}}

\def\sn{{\slashed{\nabla}}}

\def\zb{{\underline{\zeta}}}

\def\J{{\mathcal{J}}}

\def\M{{\mathcal{M}}}
\def\bT{{\textbf{T}}}

\def\bR{{\textbf{R}}}

\def\bd{{\textbf{D}}}
\def\ti{\tilde}

\def\hk{{\hat{k}}}
\def\I{{\mathcal I}}

\def\beaa{\begin{eqnarray*}}
\def\eeaa{\end{eqnarray*}}

\def\ba{\begin{array}}
\def\ea{\end{array}}
\def\be#1{\begin{equation} \label{#1}}
\def \eeq{\end{equation}}
\def\nn{\nonumber}

\def\l{\langle}
\def\r{\rangle}

\def\nn{\nonumber}
\def\S{{\mathcal S}}

\def\S2{{\Bbb S}^2}

\def\gb{\underline{g}}

\def\Lb{\underline{L}}

\def\tr{\mbox{tr}}

\def\H{{\mathcal H}}

\def\N{{\mathcal N}}

\def\c{\cdot}

\def\a{\alpha}
\def\b{\beta}

\def\l{\langle}
\def\r{\rangle}
\def\ga{\gamma}
\def\Ga{\Gamma}

\def\p{\partial}

\def\nab{\nabla}
\def\nabb{\underline{\nabla}}

\def\Lb{{\underline{L}}}
\def\aaa{{\mathbf a}}

\def\tr{\mbox{tr}}
\def\Tr{\mbox{Tr}}

\def\Nb{{\underline{\N}}}
\def\bN{{\mathbf{N}}}

\def\tir{{\tilde r}}

\def\f14{\frac{1}{4}}
\def\f12{{\frac{1}{2}}}

\def\t1a{t^{-\frac{1}{a}}}

\def\bm{{\bf m}}

\def\sl{\slashed}

\def\sD{\slashed{\Delta}}

\def\ckk{\check}

\newcommand{\bea}{\begin{eqnarray}}
\newcommand{\eea}{\end{eqnarray}}

\def\nn{\nonumber}
\def\be{{(e)}}

\def \up#1{{}^{(#1)}\!}

\newcommand{\chib}{\underline{\chi}}

\begin{document}
\title[]
{A geometric perspective on the method of descent}
\author{Qian Wang}
\address{
Oxford PDE center, Mathematical Institute, University of Oxford, Oxford, OX2 6GG, UK}
  \email{qian.wang@maths.ox.ac.uk}
  \date{\today}
\maketitle
\begin{abstract}
We derive a representation formula for the tensorial wave equation $\Box_\bg \phi^I=F^I$ in globally hyperbolic
Lorentzian spacetimes $(\M^{2+1}, \bg)$ by giving a geometric formulation of the method of descent which is
applicable for any dimension.
\end{abstract}

\section{\bf Introduction}

We consider the wave equation on $(d+1)$ globally hyperbolic, smooth Lorentzian spacetimes $(\M,\bg)$.
In the case that $(\M,\bg)$ is the Minkowski spacetime $({\mathbb R}^{d+1},\bm)$, a representation formula
for the solutions of wave equation can be obtained by the classical theory (\cite{Evans}). In fact, when $d$
is odd, one can derive the formula by reducing the problem to a wave equation in  $({\mathbb R}^{1+1},\bm)$
via the spherical mean and then apply the d'Alembert's formula; when $d$ is even, one can derive the formula
by considering a wave equation  in $({\mathbb R}^{(d+1)+1}, \bm)$ and then apply the method of descent due to
Hadamard. In particular, for $d=3$ and $d=2$, the corresponding formulae are called the Kirchhoff
formula and the Poisson formula respectively.

When $(\M,\bg)$ is a globally hyperbolic $(3+1)$ Lorentzian spacetime, a geometric Kirchhoff formula is provided
in \cite{KSob} for the tensorial wave equations $\Box_{\bg} \phi^I=F^I$. In this paper, we give the geometric
formulation of the method of descent in $(2+1)$ Lorentzian spacetimes. By using this formulation, we obtain a
first order, intrinsic,  representation formula in physical space for the solutions of  tensorial wave equations in  $(2+1)$ spacetimes.
Our construction is purely geometric, which potentially can be used in quasi-linear problems, such as $(2+1)$ gravity,
when the geometric quantities appeared in the formula have better structures due to the curvature properties of
the background geometry.

There are various types of  parametrix for wave equations in the curved spacetime. When establishing the Strichartz
estimates or the bilinear estimates for solving well-posedness problem with large rough data, one may use the Fourier
type parametrix, see \cite{KRpara} and \cite{Sm} for examples.  The Kirchhorff parametrix constructed in \cite{KSob}
is used in \cite{KR2} and \cite{Wang10} to provide  geometric breakdown criteria for the solutions of $(3+1)$ Einstein
vacuum equations with large data.  The application can be traced back to the work of \cite{EM2}, where the authors
prove the global existence result for Yang-Mills-Higgs equations. A crucial step of the proof is to use the representation
formula for the wave equation in Minkowski space to represent and control the curvature. This strategy is later used
in \cite{chsh} to prove the same result in globally hyperbolic spacetimes, where they employed the Hadamard parametrix
in \cite{Fried} in the curved spacetimes. This parametrix at a point $p$ is in physical space, supported within the
domain of dependence, nevertheless it is not purely supported on the boundary of the casual past of the point,
in $(3+1)$ spacetime, due to a series of corrections by using transport equations. Moreover, this parametrix is ill-suited
for nonlinear problems because it requires the property of  geodesic convexity and infinite smoothness on the metric
for controlling the correction terms.  The Kirchhorff formula established in \cite{KSob} in particular is supported
only on the null boundary of the causal past, i.e. the backward lightcone of the point $p$, which coincides with the
Huygens principle. In fact, this formula is an integral purely along the null boundary. The null boundary and the
quantities on the background geometry involved in the formula have much better regularity property, which can be
controlled in terms of the Bel-Robinson energy flux by using a series of sharp trace estimates in Einstein spacetime.
This advantage is  very crucial for the applications in \cite{KR2} and \cite{Wang10}.

In a $(2+1)$ curved spacetime,  there is no such formula available. To implement the method of descent based on
the Kirchhorff formula in a $(3+1)$ Lorentzian spacetime, we need to establish the geometric correspondence between
the geometry of the  lightcone of a point  in the $(3+1)$ spacetime with the causal past of the same point in the $(2+1)$ spacetime.

With $\rho$  the Lorentzian distance inside the backward lightcone in $(d+1)$ spacetime, we observe that the vector
field  $-\bd \rho$ in $\T \M$ corresponds to the null geodesic generator $\ti L$ in the corresponding backward light
cone  in $((d+1)+1)$ spacetime. This allows  us  to express the geometric quantities and the null frames on the light
cone in $((d+1)+1)$ spacetime in terms of the hyperboloidal frames  in $(d+1)$ spacetime. These quantities include the
null expansion of the light rays in $((d+1)+1)$ spacetime and the area expansion of the timelike geodesic congruence
in $(d+1)$ and other connection coefficients of these frames. Based on this observation, we can uncover the relation
between  the radius of injectivity of the corresponding geodesic congruences in two spacetimes of different dimensions.
This relation is in particular important since the Kirchhorff formula in \cite{KSob} holds within the null radius of injectivity.

In this paper, we focus on the case that $d=2$, while our method applies to higher dimensions.  As long as the analogous
representation formulae of \cite{KSob} in other odd dimensions are available, we can similarly obtain the formulae in
the even dimensions. This paper is organized as follows. In Section \ref{sec2.1}, we give the geometric set-up and the
main theorem of the paper. In Section \ref{sec2.2}, we give the relation between the geometry of  the null cones with
the vertex $p$ in $(3+1)$ spacetime and the causal past of the same vertex
in $(2+1)$ spacetime.  By  uncovering the  quantitative correspondence  between connection coefficients of the null
frames in $(3+1)$ spacetime and those of the triads in $(2+1)$ spacetime, we control the null radius of injectivity
in terms of the causal radius of injectivity in $(2+1)$ spacetime. In Section 2.3 we then complete the proof of the
main result. This result, in the flat case, coincides with the
Poisson formula in  Minkowski space. In Appendix, we give a proof of the Kirchhoff formula used in Section 2.

\section{\bf A geometric method of descent}

Let $(\M,\bg)$ be a $(2+1)$-globally hyperbolic smooth Lorentzian spacetime. We assume that $\M$ is
foliated by a time function $t$ and the metric $\bg$ takes the form\begin{footnote}{Throughout the paper
we use the Einstein summation convention. We set $x^0=t$. A little Greek letter is used to denote an index from $\{0,1,2\}$ and
a little Latin letter is used to denote an index from $\{1, 2\}$, e.g. $\a=0, 1, 2$ and $i = 1, 2$.}
\end{footnote}
\begin{equation*}
\bg = \bg_{\a\b} dx^\a dx^\b = -n^2 dt^2+g_{ij}dx^i dx^j,
\end{equation*}
where $n$ is the lapse function and $g = g_{ij} dx^i dx^j$ are Riemannian metrics on $\Sigma_t$, the level sets of
the time function $t$.

Let $\bd$ denote the covariant differentiation on $(\M, \bg)$ and let $\Box_\bg:=\bg^{\a\b}\bd_\a\bd_\b$
denote the induced d'Alembertian. Consider the tensorial wave equation $\Box_\bg \phi_I = F_I$ in $(\M, \bg)$.
In this paper we will develop a
geometric formulation of the method of descent to derive a representation formula for $\phi_I$ which can be viewed
as an extension of the classical Poisson formula for the scalar wave equation in the Minkowski spacetime
$({\mathbb R}^{2+1}, \bm)$.

\subsection{\bf Set-up and main result}\label{sec2.1}

Let $\bT$ be the future directed time-like unit normal of  $\Sigma_t$. 
Any point in $(\M,\bg)$ can be written as $(t,x)$ where $x\in \Sigma_t$. Given $p\in \M$, 
we denote by $\I^-(p)$, $\J^-(p)$ and $\N^-(p)$ the chronological past, the causal past and the backward
light-cone in $(\M,\bg)$ initiating from $p$. Note that $\N^-(p)$ is a surface ruled by the backward null
geodesics from $p$. In the sequel, by $\Sigma_t$ we mean $\Sigma_t\cap \I^-(p)$.

For a fixed point $p\in \M$, we consider 
\begin{equation*}
\mathbb{H}^2:= \left\{V\in \T_p\M: \| V\|_{\bg(p)}=-1, V^0= \bg( V, \bT)>0\right\}.
\end{equation*}
Relative to a geodesic normal coordinate at $p$, we can regard
$$
\mathbb{H}^2= \left\{ V=(V^0, V^1, V^2): (V^0)^2 - \sum_{i=1}^2 (V^i)^2 =1 \mbox{ and } V^0>0\right\}
$$
which is the canonical hyperboloid in ${\mathbb R}^{2+1}$. For each $V\in \mathbb{H}^2$ let $\Upsilon_V(\rho)$
be the time-like geodesic with $\Upsilon_V(0)=p$ and $\Upsilon_V'(0) = V$, and let $\rho(t)$ denote the
Lorentzian distance from $p$ to the intersection point of $\Upsilon_V(\rho)$ with $\Sigma_t$. Note that
$\rho(t)$ is a function not only depending on $t$ but also on $V$; we suppress $V$ for simplicity.
We then define the past time-like radius of injectivity $\d_*$ at $p$ in $(\M, \bg)$ to be the supremum
over all the values $\tau>0$ for which the exponential map
\begin{equation}\label{12.29.4}
\exp_p: (t, V)\to \Upsilon_V(\rho(t))
\end{equation}
is a global diffeomorphism from $(t_p-\tau, t_p)\times \mathbb{H}^2$ to its image in $\I^-(p)$.
In this paper, we only consider the part of $\I^-(p)$ within the time-like
radius of injectivity, which will be still denoted as $\I^-(p)$ by abuse of notation.

For $(t,x)\in \J^-(p)$ let $\rho(t,x)$ be the Lorentzian distance to $p$ in $\J^-(p)$.  Clearly,
$\rho(t,x)=0$ iff $(t,x) \in \N^-(p)$. Moreover, within $\J^-(p)$ with $0<t_p-t<\d_*$ this function is smooth
and verifies
\begin{equation}\label{2.13.6}
\bg^{\a\b}\p_\a \rho \p_\b \rho=-1,\qquad \rho(p)=0.
\end{equation}
In $\I^-(p)\subset\M$ we define the vector field $\fB$ by
$$
\fB: =-\bd \rho  =-\bg^{\a\b} \p_\a \rho \p_\b.
$$
Then $\fB$ is geodesic, i.e. $\bd_\fB \fB=0$ and satisfies $\bg( \fB, \fB)=-1$. Moreover,
\begin{equation}\label{8.6.1}
\fB = \left(d \exp_p\right)_{\rho V} (\p_\rho).
\end{equation}
Let $H_\rho$ denote the level sets of $\rho$.
Then $\fB$ is the past directed unit normal of $H_\rho$ and is the generator of the timelike geodesic $\Upsilon_V(\rho)$.

We define the frame lapse $\bb$ by
\begin{equation}\label{eq_1}
\bg( \fB, \bT)=\bb^{-1}\frac{t_p-t}{\rho}.
\end{equation}
Let $\tau := t_p-t$. Then by noting that $\bT = n^{-1} \p_t$, we have from (\ref{eq_1}) that
\begin{equation}\label{eq_2}
\fB(\tau)=n^{-1} \bb^{-1}\frac{\tau}{\rho}.
\end{equation}

Let $g$ be the induced metric on $\Sigma_t$ and let $\nab$ be the corresponding covariant derivative on $\Sigma_t$.
We consider the lapse function $a^{-1}:=|\nabla \rho|_g$. By using (\ref{eq_1}) we have
\begin{equation}\label{2.13.5}
-\bT(\rho)=\frac{\bb^{-1}\tau}{\rho}.
\end{equation}
This together with (\ref{2.13.6}) then implies that
\begin{align*}
-1=\bg^{\a\b}\p_\a \rho\p_\b\rho&=-(\bT(\rho))^2+g^{ij} \p_i \rho\p_j \rho
= -\frac{\bb^{-2}\tau^2}{\rho^2}+ |\nabla \rho|_g^2.
\end{align*}
Hence the lapse $a$ can be written as
\begin{equation*}
a^{-2} = |\nabla \rho|_g^2 = \frac{\bb^{-2}\tau^2}{\rho^2}-1,
\end{equation*}
which also implies $\bb^{-1}\tau\ge \rho$ in $\I^-(p)$. By setting $\ti r=\sqrt{\bb^{-2} \tau^2-\rho^2}$, we have
\begin{equation}\label{aa1}
a^{-1}=\frac{\tir}{\rho}.
\end{equation}
Let $S_{t,\rho} := H_\rho\cap \Sigma_t$. Then for each fixed $t$, $\{S_{t,\rho}\}_\rho$ is a family of 1-dimensional curves
diffeomorphic to circles and forms the radial foliation of $\Sigma_t$. Let $\bN$ be the radial normal of $S_{t,\rho}$
in $\Sigma_t$. Then
\begin{equation}\label{2.13.4.1}
\bN = -\frac{\nabla \rho}{|\nabla \rho|_g} = -a\nab \rho.
\end{equation}
In view of (\ref{2.13.5}) and (\ref{2.13.4.1}), we can decompose $\fB$ in terms of $\bT$ and $\bN$ as
\begin{equation}\label{fb1}
\fB=-\frac{\bb^{-1}\tau}{\rho}\bT+a^{-1} \bN.
\end{equation}
We set
\begin{equation}\label{fb2}
 \underline{\fB}=-\frac{\bb^{-1}\tau}{\rho}\bT-a^{-1}\bN.
 \end{equation}
 Clearly
\begin{equation}\label{2.12.1}
\bg(\fB, \fB)= \bg(\underline{\fB}, \underline{\fB})=-1.
\end{equation}

Let $\gb$ be the induced metric of $\bg$ on $H_\rho$ and $\nabb$ be the Levi-civita connection of $\gb$.
By introducing the projection tensor
$$
\ckk{\Pi}_{\a\b}=\bg_{\a\b}+\fB_\a\fB_\b,
$$
we have
\begin{equation*}
\nabb^\a=\ckk{\Pi}_{\b\ga}\bg^{\a\b}\bd^\ga \quad \mbox{ and } \quad
|\nabb \tau|_{\gb}=(an)^{-1}.
\end{equation*}
Let $\Nb$ be the  radial normal of  $\{S_{\tau,\rho}\}_\tau\subset H_\rho$. Then we have
$$
\Nb= \frac{\nabb \tau}{|\nabb \tau|_{\gb}}=an \nabb \tau.
$$
Similar to \cite[Page 13]{Wang16}, $\Nb$ can be decomposed as
\begin{equation}\label{7.12.12.16}
\Nb=-\frac{\tir}{\rho}\bT+\frac{\bb^{-1}\tau}{\rho}\bN.
\end{equation}

We will use $\sn$ to denote the Levi-civita connection of the induced metric on $S_{\tau, \rho}$
and use $e_\sl{A}$ to denote a unit tangent vector field on $S_{\tau, \rho}$.


\begin{definition}\label{7.21.1.0}

\begin{enumerate}[leftmargin = 0.8cm]
\item
We denote by $\pmb{\pi}$ the second fundamental form of $(\Sigma_t,g)\subset(\M, \bg)$, i.e.
\begin{equation*}
\pmb{\pi}(X, Y) = -\bg({\bf D}_X {\bf T}, Y)
\end{equation*}
for $X, Y\in \T \Sigma_t$. The trace of $\pmb{\pi}$ is $\emph{\Tr} \pmb{\pi}=g^{ij} \pmb{\pi}_{ij}$.

\item We denote by $k$ the second fundamental form of $H_\rho\subset (\M, \bg)$, i.e.
\begin{equation}\label{k1}
k(X,Y)=\bg({\bf D}_X\fB, Y)
\end{equation}
for $X, Y\in \T H_{\rho}$ in $(\M, \bg)$. We denote the trace and traceless part of $k$ by $\emph{\tr} k$ and $\hk$ respectively.
Note that in Minkowski space $({\mathbb R}^{2+1}, \bm)$ we have $\emph{\tr} k=\frac{2}{\rho}$.

\item We introduce the connection coefficients
$$
\omega := - \fB\left(\frac{\rho}{\bb^{-1} \tau}\right) \quad \mbox{ and } \quad
\zb_\sl{A}: = \l {\bf D}_\fB \Nb, e_\sl{A}\r.
$$
\end{enumerate}
\end{definition}

We first give some preliminary results on the geometric quantities defined above.

\begin{lemma}
For the frame lapse $\bb$ and $a$ and the connection coefficients $\omega$ and $\zb$, there hold
\begin{align}
\zb_\sl{A}& = \frac{\rho}{\tir} \l {\bf D}_\fB {\bf T}, e_\sl{A}\r,  \label{7.10.16.3}\\
\omega &= \frac{\rho \tir}{\bb^{-2}\tau^2} \l {\bf D}_\fB {\bf T}, {\bf N}\r, \label{7.15.6}\\
\fB(\bb^{-1}) & = \frac{\bb^{-1}}{\rho}(1-\bb^{-1}n^{-1}) + \frac{\bb^{-2}\tau}{\rho}\omega, \label{7.12.8.16}\\
\sn_\sl{A}\log a&=\frac{\bb^{-1}\tau}{\tir}(\pmb{\pi}_{\bN \sl{A}}-k_{\sl{A}\Nb}), \label{8.3.3}\\
\l {\bf D}_{\bf T} \bN, e_\sl{A}\r&=k_{\Nb \sl{A}}+\frac{\bb^{-1}\tau}{\tir}\l {\bf D}_{\bf T} {\bf T}, e_\sl{A}\r. \label{8.3.5}
\end{align}
\end{lemma}

\begin{proof}
We first derive (\ref{7.10.16.3}). From (\ref{fb1}) and (\ref{7.12.12.16}) it follows that
\begin{equation}\label{t1}
\bT=-\frac{\bb^{-1}\tau}{\rho} \fB+\frac{\tir}{\rho}\Nb
\end{equation}
and hence $\Nb=\frac{\rho}{\tir}(\bT+\frac{\bb^{-1}\tau}{\rho}\fB)$. Consequently, by using $\bd_\fB \fB=0$ we have
\begin{equation}
\zb_\sl{A} = \l \bd_\fB \Nb, e_\sl{A}\r = \frac{\rho}{\tir} \l \bd_\fB \bT, e_\sl{A}\r + \frac{\bb^{-1}\tau}{\tir} \l \bd_\fB\fB, e_\sl{A}\r
 = \frac{\rho}{\tir}\l \bd_\fB \bT, e_\sl{A}\r.
\end{equation}

To see (\ref{7.15.6}), we use (\ref{eq_1}), (\ref{fb1}) and  $\bd_\fB \fB=0$ to obtain
\begin{align*}
\omega & = \frac{\rho^2}{\bb^{-2}\tau^2} \fB\left(\frac{\bb^{-1} \tau}{\rho}\right)
= \frac{\rho^2}{\bb^{-2}\tau^2} \fB (\l\fB, \bT\r) = \frac{\rho^2}{\bb^{-2}\tau^2} \l \fB, \bd_\fB \bT\r\\
& = \frac{\rho^2}{\bb^{-2}\tau^2} a^{-1} \bg(\bd_\fB \bT, \bN) = \frac{\rho \tir}{\bb^{-2}\tau^2} \bg(\bd_\fB \bT, \bN).
\end{align*}

To obtain (\ref{7.12.8.16}), we use $\bb^{-1}=\frac{\rho}{\tau}(\fB, \bT)$. By using (\ref{eq_2}) we have
\begin{align*} 
\fB\left(\frac{\rho}{\tau}\right) = \frac{1}{\tau} - \frac{\rho}{\tau^2} \fB(\tau) = \frac{1}{\tau}(1-n^{-1}\bb^{-1}).
\end{align*}
Therefore, in view of (\ref{eq_1}), $\bd_\fB \fB = 0$ and (\ref{fb1}), it follows that
\begin{align*}
\fB(\bb^{-1}) &= \fB\left(\frac{\rho}{\tau}\l\fB, \bT\r\right)
= \fB\left(\frac{\rho}{\tau}\right)\l\fB, \bT\r + \frac{\rho}{\tau} \l \fB,\bd_\fB\bT\r \\
& = \frac{\bb^{-1}}{\rho}(1-\bb^{-1}n^{-1})+ \frac{\tir}{\tau} \l \bN, \bd_\fB \bT\r.
\end{align*}
In view of (\ref{7.15.6}), we therefore obtain (\ref{7.12.8.16}).

To obtain (\ref{8.3.3}), we first use $\tir^2 = \bb^{-2}\tau^2 - \rho^2$ to derive that $\sn_\sl{A} \tir =
\frac{\bb^{-1} \tau}{\tir} \sn_\sl{A} (\bb^{-1} \tau)$. Thus, in view of (\ref{eq_1}) we have
\begin{align*} 
\sn_\sl{A} \log \tir &= \frac{\bb^{-1}\tau}{\tir^2}\sn_\sl{A} (\bb^{-1}\tau)
= \frac{\bb^{-1}\tau \rho}{\tir^2}\sn_\sl{A} \left(\frac{\bb^{-1}\tau}{\rho}\right)
=\frac{\bb^{-1} \tau \rho}{\tir^2} \sn_\sl{A} \l \fB, \bT\r \\
& =\frac{\bb^{-1}\tau\rho}{\tir^2} \left(\l \bd_\sl{A} \fB, \bT\r + \l \fB, \bd_\sl{A} \bT\r\right).
\end{align*}
By using (\ref{t1})  and (\ref{fb1}) we can further obtain
\begin{equation*}
\sn_\sl{A} \log \tir = \frac{\bb^{-1} \tau}{\tir} \left(\l \bd_\sl{A} \fB, \Nb\r + \l \bN, \bd_\sl{A} \bT\r \right)
=\frac{\bb^{-1} \tau}{\tir} (k_{\sl{A}\Nb}-\pmb{\pi}_{\sl{A}\bN}).
\end{equation*}
 (\ref{8.3.3}) then  follows  by using (\ref{aa1}) and the above identity.

Finally, we prove (\ref{8.3.5}). From (\ref{fb1}) we have $\bN = \frac{\rho}{\tir} \fB + \frac{\bb^{-1}\tau}{\tir} \bT$.
Thus, by using $\l \fB, e_\sl{A}\r = \l \bT, e_\sl{A}\r =0$ we have
\begin{align*}
\l \bd_\bT \bN, e_\sl{A}\r =\frac{\rho}{\tir} \l\bd_\bT \fB, e_\sl{A}\r+\frac{\bb^{-1}\tau}{\tir} \l \bd_\bT \bT, e_\sl{A}\r.
\end{align*}
By using (\ref{t1}) and $\bd_\fB \fB=0$ we obtain
\begin{align*}
\l \bd_\bT \bN, e_\sl{A}\r = \l \bd_\Nb \fB, e_\sl{A}\r + \frac{\bb^{-1}\tau}{\tir} \l \bd_\bT \bT, e_\sl{A}\r
= k_{\Nb \sl{A}}+\frac{\bb^{-1}\tau}{\tir}\l \bd_\bT \bT,e_\sl{A}\r.
\end{align*}
The proof is therefore complete.
\end{proof}

Now we are ready to state the main result of this paper.
\begin{theorem}[Main theorem]\label{7.13.2.16}
Consider a tensorial wave equation
\begin{equation}\label{2.14.5}
\Box_\bg \phi_I=F_I
\end{equation}
on $(\M, \bg)$. Let $p$ be any point in $(\M, \bg)$ and $t_0$ verify $0<t_p-t_0<c_*(p,t).$\begin{footnote}{The definition of the causal radius of injectivity  $c_*(p,t)$ is given in Theorem \ref{8.7.2}.}\end{footnote} Denote by $\I_*^-(p)$ the interior of the
backward lightcone from $p$  with $t\in [t_0, t_p]$. Given a tensor $J$ at $p$ of the same type as $\phi_I$,
let $A_I$ be a tensor field on $\I^-_*(p)$ satisfying
\begin{equation}\label{tsa2}
{\bf D}_\fB A_I+ \left(\f12 \emph{\tr} k+\frac{n^{-1}\bb^{-1}-1}{\rho}\right) A_I=0, \quad \lim_{t\rightarrow t_p} \tau A_I=J.
\end{equation}
Then there holds
\begin{equation}\label{7.19.1}
2\pi (n\bg( \phi, J))(p)=-\int_{\I_*^-(p)} F_I A^I \frac{\tau}{\rho}  nd\mu_{\Sigma_t} dt+\I_1 +\I_2+\I_3,
\end{equation}
with
\begin{align*}
\I_1 & =\int_{\Sigma_{{t_0}}\cap {\I_*^-(p)}} \left[-{\bf D}_{\underline{\fB}} \phi_I+\f12\phi_I
\left( \emph{\tr} k-2\frac{\tau}{\bb\rho} (\emph{\Tr} \pmb{\pi}
-\frac{\bb^2\tir^2}{\tau^2} \pmb{\pi}_{\bN\bN})\right)\right]A^I d\mu_{\Sigma_{t_0}}, \displaybreak[0]\\
\I_2 &= -2 \int_{\I_*^-(p)}\left[ \frac{\tir}{\rho} \zb^\sl{A} {\bf D}_\sl{A} \phi_I
+ \frac{\rho}{\tir} \omega {\bf D}_\bN \phi_I\right] A^I \bb n d \mu_{\Sigma_t} d t  \displaybreak[0]\\
& \quad \, - \int_{\I_*^-(p)} \left[\sn(\bb A^I )\sn \phi_I +\frac{\rho^2}{\bb^{-2}\tau^2} {\bf D}_\bN(\bb A^I) {\bf D}_\bN\phi_I\right]
\frac{\bb^{-1}\tau n}{\rho}d\mu_{\Sigma_t}dt, \displaybreak[0]\\
\I_3 & = \int_{\I_*^-(p)} \left[\frac{\tir}{\rho} ({\bf R}*\phi)_I +  \omega \emph{\tr} k \phi_I
+ \f12\frac{\rho}{\bb^{-1} \tau} ({\bf R}_{\fB\fB} + |\hat k|^2) \phi_I\right] A^I \bb n d \mu_{\Sigma_t} dt \displaybreak[0]\\
& \quad \, +\int_{\I_*^-(p)} \left(\fB + \frac{\emph{\tr} k}{2} - \frac{\bb^{-1} \tau}{\rho} \omega\right)
\left(\emph{\Tr} \pmb{\pi} - \frac{\tir^2}{\bb^{-2} \tau^2} \pmb{\pi}_{\bN\bN}\right) \phi_I A^I \bb n d\mu_{\Sigma_t}dt,
\end{align*}
where, for $I = \{\mu_1, \cdots, \mu_l\}$,
\begin{align*}
({\bf R} * \phi)_I = \sum_{i=1}^l \tensor{{\bf R}}{_{\mu_i}^\a _{ {\bf T}\bN}}
\tensor{\phi}{_{\mu_1\cdots\mu_{i-1} \a \mu_{i+1}\cdots \mu_l}}.
\end{align*}
\end{theorem}

As a simple application, we will use the representation formula in Theorem \ref{7.13.2.16} to recover the
Poission formula for the scalar linear wave equation  $\Box_{\bm}\phi=F$ in the (2+1)-Minkowski space-time
$({\mathbb R}^{2+1}, \bm)$, with Cauchy data given at $t=0$.  Let $p$ be a point in $({\mathbb R}^{2+1}, \bm)$
with coordinates $(x_p, t_p)$ and $t_p>0$. Note that $n=1$, $\I_*^-(p) = \{(x, t) \in {\mathbb R}^{2+1}: 0\le t< t_p-|x-x_p|\}$
and  $\rho(x, t) = \sqrt{\tau^2-r^2}$ in $\I_*^-(p)$ with $\tau = t_p-t$ and $r = |x-x_p|$. We can derive that
\begin{align*}
& \fB = -\frac{\tau}{\rho} \p_t + \frac{r}{\rho} \p_r, \quad \fBb = -\frac{\tau}{\rho} \p_t -\frac{r}{\rho} \p_r,
\quad \bb = 1, \quad \tir = r, \quad \omega =0,\\
& \pmb{\pi}=0, \quad  \zb =0, \quad \bR = 0, \quad \tr k(x, t) = \frac{2}{\rho(x, t)}, \quad \hat k =0
\end{align*}
For $J=1$ we can see that $A = \tau^{-1}$. Consequently $\I_2=\I_3 =0$ and it follows from Theorem \ref{7.13.2.16} that
\begin{align*}
2\pi \phi(p) &= -\int_0^{t_p} \int_{|x-x_p|< t_p-t} \frac{F(x,t)}{\sqrt{(t_p-t)^2-|x-x_p|^2}} dx dt \\
& \quad + \frac{1}{t_p} \int_{|x-x_p|<t_p} \frac{t_p \p_t\phi(x,0)+r\p_r\phi(x, 0)+ \phi(x,0)}{\sqrt{t_p^2-|x-x_p|^2}} dx.
\end{align*}
Hence in the Minkowski space-time $({\mathbb R}^{2+1}, \bm)$, Theorem \ref{7.13.2.16} gives the classical Poisson formula.

\subsection{A Kirchhoff formula in 3-dimensional space-time}\label{sec2.2}

We will give the proof of Theorem \ref{7.13.2.16} by a geometric method of descent. To this end, we use $(\M, \bg)$ to introduce
the manifold $\widetilde \M=\M\times {\mathbb R}$ and a Lorentzian metric $\tilde \bg$ on $\widetilde \M$ by\begin{footnote}
{We will identify $z$ with $x^3$. Besides the convention on page 3, a Greek letter with tilde is used to denote
an index from $\{0, 1, 2, 3\}$, e.g. $\tilde \a = 0, 1, 2, 3$.}\end{footnote}
\begin{equation}\label{g_1}
\tilde \bg = \tilde\bg_{\tilde\a\tilde\b} dx^{\tilde\a}dx^{\tilde\b}:=\bg_{\a\b} dx^\a dx^\b+d z^2.
\end{equation}
We use $\widetilde\bd$ to denote the Levi-Civita connection of $\tilde \bg$ on $\widetilde \M$.

We may identify $\M$ with $\M\times\{0\}$ as a submanifold of $\widetilde \M$. For a function or tensor on $(\M, \bg)$
we may use a standard procedure to extend it to a function or tensor on $(\widetilde \M, \tilde \bg)$ such that it is independent
of $z$ with vanishing $\p_z$-components; such extensions are called $\M$-tangent extensions and are denoted
by the same notation. Let $\Box_{\tilde \bg}:= \tilde \bg^{\tilde \a \tilde \b} \widetilde \bd_{\tilde \a} \widetilde \bd_{\tilde \b}$
be the d'Alembertian with respect to $(\widetilde \M, \tilde \bg)$. Then for the tensor fields $\phi$
and $F$ satisfying $\Box_\bg \phi =F$ in $\M$, we have $\Box_{\tilde \bg} \phi = F$ in $\widetilde \M$.
Therefore, to derive a representation formula of $\phi$ in $(\M, \bg)$, we will use a Kirchhoff formula in $(\widetilde \M, \tilde \bg)$.
We start with some preparation.

\begin{lemma}
Let $\widetilde \Ga$ and $\Ga$ denote the Christoffel symbols of $\tilde\bg$ and $\bg$ respectively. There hold
\begin{align}
&\widetilde \Ga_{z\ga}^\a = \widetilde \Ga_{\a\ga}^{z} = \widetilde \Ga_{\a z}^z
= \widetilde \Ga_{zz}^z = \widetilde \Ga_{zz}^\a =0; \quad
\widetilde \Ga^\ga_{\a\b} = \Ga^\ga_{\a\b}, \label{chrisplit}\\
&\widetilde {\bf D}_\fB \fB = \widetilde {\bf D}_{\p_z}\fB = \widetilde {\bf D}_{\fB} \p_z = \widetilde {\bf D}_{\p_z} \p_z= 0. \label{7.10.16.4}
\end{align}
\end{lemma}

\begin{proof}
Note that $\tilde \bg_{z\a} =0$, $\tilde \bg_{zz} =1$ and $\tilde \bg_{\tilde \a \tilde \b}$ is independent of $z$.
We can obtain (\ref{chrisplit}) directly from the formula
\begin{align*}
\widetilde{\Ga}^{\tilde\a}_{\tilde \b \tilde \ga}
=\f12 {\tilde \bg}^{\tilde\a \tilde \eta} \left( \p_{\tilde \b} \tilde \bg_{\tilde \eta \tilde \ga}
+ \p_{\tilde \ga} \tilde \bg_{\tilde \b \tilde \eta}- \p_{\tilde \eta} \tilde \bg_{\tilde \b \tilde \ga}\right).
\end{align*}

Next we show (\ref{7.10.16.4}). Note that
$$
(\widetilde \bd_\fB \fB)^{\tilde \b} = \fB^{\tilde \a} \widetilde \bd_{\tilde \a} \fB^{\tilde \b}
= \fB^{\tilde \a} \left( \p_{\tilde \a} \fB^{\tilde \b} + \widetilde \Ga_{\tilde \a \tilde \eta}^{\tilde \b} \fB^{\tilde \eta}\right).
$$
Since $\tilde \bg(\fB, \p_z) =0$ and $\fB$ is independent of $z$, we may use (\ref{chrisplit}) to obtain $(\widetilde \bd_\fB \fB)^z=0$ and
$$
(\widetilde \bd_\fB \fB)^\b = \fB^\a \left( \p_\a \fB^\b + \Ga_{\a \eta}^\b \fB^\eta\right) = (\bd_\fB \fB)^\b =0,
$$
where for the last equality we used the fact that $\fB$ is geodesic. The remaining three equalities in (\ref{7.10.16.4})
can be proved similarly.
\end{proof}

For $p\in \M$, let $\widetilde \N^-(p)$ denote the backward light cone with vertex $p$ in $(\widetilde \M, \tilde \bg)$.
Then $\N^-(p)$ can be identified as a subset of $\widetilde \N^-(p)$ that is ruled by null geodesics in
$(\widetilde \M, \tilde \bg)$ with vanishing $z$-coordinate. For $(t, x, z)\in \widetilde \M$, let
$$
u =u(t, x, z):= |z| - \rho(x, t).
$$

\begin{lemma}\label{lem17.1.7}
Within $0<\tau:=t_p-t <\d_*$, the level set $\{u=0\}$ of $u$ in $(\widetilde \M, \tilde \bg)$ coincides with the
backward null cone $\widetilde\N^-(p)$ with vertex $p$.
\end{lemma}

\begin{proof}
By the geodesic equation and (\ref{chrisplit}) it is easy to see that $\widetilde \Upsilon(s) := (\Upsilon(s), z(s))$
is a geodesic in $(\widetilde \M, \tilde \bg)$ with $\widetilde \Upsilon(0) = p$ if and only if $\Upsilon(s)$ is a
geodesic in $(\M, \bg)$ with $\Upsilon(0)=p$ and $z(s) = c s$ for some constant $c$.

If $\widetilde \Upsilon(s) := (\Upsilon(s), z(s))$ is a null geodesic in $(\widetilde \M, \tilde \bg)$ initiating from $p$,
then
$$
0 = \tilde \bg(\widetilde \Upsilon'(s), \widetilde \Upsilon'(s)) = \bg(\Upsilon'(s), \Upsilon'(s))+ |z'(s)|^2
$$
which implies that
$$
|z'(s)| = \left(-\bg(\Upsilon'(s), \Upsilon'(s))\right)^{1/2}.
$$
Integrating this equation with respect to $s$ and using the definition of the Lorentzian distance $\rho$, we can obtain
$$
0 = \int_0^s \left(|z'(\tilde s)| - \left(-\bg(\Upsilon'(\tilde s), \Upsilon'(\tilde s))\right)^{1/2}\right) d \tilde s
= |c| s - \rho(\Upsilon(s)) = |z(s)| - \rho(\Upsilon(s)).
$$
This shows that, within $\tau<\d_*$, $\widetilde\Upsilon \subset \{u=0\}$ and hence $\widetilde\N^-(p) \subset \{u=0\}$.

Conversely, let $(t,x, z)$ be any point on $\{u=0\}$ with $\tau=t_p-t<\d_*$. If $z=0$ then $\rho(t, x) =0$ and hence
$(t, x, z) \in \N^-(p)\subset \widetilde \N^-(p)$. Thus we may assume $z\ne 0$. We can find a time-like geodesic
$\Upsilon(s)$ in $(\M, \bg)$ initiating from $p$, with $\Upsilon(s_0) = (t, x)$ for some $s_0>0$. Set
$$
\widetilde \Upsilon(s) = (\Upsilon(s), \frac{z}{s_0} s).
$$
Then $\widetilde \Upsilon(s_0) = (t, x, z)$ and $\widetilde \Upsilon(s)$ is a geodesic in $(\widetilde \M, \tilde \bg)$.
Moreover
$$
\tilde \bg(\widetilde \Upsilon'(s), \widetilde \Upsilon'(s)) = \bg(\Upsilon'(s), \Upsilon'(s)) + \frac{z^2}{s_0^2}.
$$
Since $\Upsilon(s)$ is a geodesic in $(\M, \bg)$, $\bg(\Upsilon'(s), \Upsilon'(s))$ is a constant and thus
$$
\rho(t,x) = \int_0^{s_0} (-\bg(\Upsilon'(s), \Upsilon'(s))^{1/2} ds = s_0 (-\bg(\Upsilon'(s), \Upsilon'(s))^{1/2}.
$$
Consequently
$$
\tilde \bg(\widetilde \Upsilon'(s), \widetilde \Upsilon'(s)) = \frac{z^2-\rho(t, x)^2}{s_0^2} =0.
$$
This shows that $\widetilde \Upsilon(s)$ is a null geodesic in $(\widetilde \M, \tilde \bg)$ and hence $(t, x, z) \in \widetilde \N^-(p)$.
Therefore, within $\tau<\d_*$ we have $\{u=0\} \subset \widetilde \N^-(p)$.
\end{proof}

According to Lemma \ref{lem17.1.7}, within $\tau<\d_*$, the null cone $\widetilde\N^-(p)$ is the union of  three parts:
$\N^-(p)$, $\H^+$ and $\H^-$, where
$$
\H^+ := \{ z = \rho(t,x), \tau<\d_*\} \quad \mbox{ and } \quad \H^- := \{z=-\rho(t,x), \tau<\d_*\}.
$$
Note that both $\H^+$ and $\H^-$ can be regarded as graphs in $(\widetilde \M, \tilde\bg)$ over
$\I^-(p)$.

\begin{lemma}\label{2.27.1.17}
Within $\tau<\d_*$, any null geodesic in $(\widetilde \M, \tilde \bg)$ initiating from $p$ lies completely in
either $\N^-(p)$, or $\H^+$, or $\H^-$.
\end{lemma}

\begin{proof}
Since the $z$-component of a geodesic in $(\widetilde \M, \tilde \bg)$ has the form $z(s)= c s$ for some constant $c$,
the result then follows according to the sign of $c$.
\end{proof}

Let $\widetilde\Sigma_t$ denote the level set of $t$ in $(\widetilde\M, \tilde\bg)$, let $\tilde g$ be the induced
metric of $\tilde \bg$ on $\widetilde \Sigma_t$, and let $\widetilde \nab$ denote the Levi-Civita connection of  $\tilde g$ on $\widetilde{\Sigma}_t$. We set $S_t := \widetilde \N^-(p)\cap \widetilde \Sigma_t$. For $\tau_0<\delta_*$,
we have
$$
\widetilde \N^-(p)\cap\{t_p-t\le \tau_0\}= \bigcup_{\tau\le\tau_0} S_{t_p-\tau}
$$
and thus $\{S_{t_p-\tau}\}_\tau$ forms a time foliation of $\widetilde\N^-(p)$. Let $\tilde N$ be
the radial normal of $S_t$ in $\widetilde \Sigma_t$. We will derive the formula for $\tilde N$. We will
only consider the half cone $\H^+$, since $\H^-$ can be treated in the same way. We first have
$$
\widetilde \nab u=\p_z-\nab \rho.
$$
Let ${\tilde a}^{-1}=|\widetilde \nab u|_{\tilde g}$. We have ${\tilde a}^{-2}=1+|\nab \rho|_g^2$. It then follows from
(\ref{aa1}) that
\begin{equation}\label{aa}
{\tilde a}^{-1}=\frac{\bb^{-1}\tau}{\rho}.
\end{equation}
Consequently
\begin{equation}\label{2.13.4}
\tilde N= \frac{\widetilde\nab u}{|\widetilde \nab u|_{\tilde g}} = {\tilde a}(\p_z-\nab \rho).
\end{equation}
Let $\tilde \ga$ denote the induced metric on $S_t$. We use $\widetilde \sn$ to denote the Levi-Civita connection
of $\tilde\ga$ and use $\widetilde{\slashed{\Delta}}$ to denote the corresponding Laplace-Beltrami operator. By setting $v_t=\sqrt{|\tilde{\ga}|}/\sqrt{|\ga_{{\Bbb S}^2|}}$, we have $d\mu_{S_t}
= v_t d\mu_{{\mathbb S}^2}$. Since $S_t$ can be viewed as a graph over $\Sigma_t$ locally, we have
\begin{equation}\label{ael}
d\mu_{S_t} = \sqrt{1+|\nabla \rho|_g^2} d\mu_{\Sigma_t}={\tilde a}^{-1} d\mu_{\Sigma_t}.
\end{equation}

Now we introduce the null frame
\begin{equation*}
\tilde L=-\bT+\tilde N,\quad \tilde \Lb=-\bT-\tilde N
\end{equation*}
on $\widetilde\N^-(p)\cap \{\tau<\delta_*\}$. Clearly, $\l \tilde L, \tilde\Lb\r=-2$.
In view of (\ref{2.13.4}), we have on $\H^+$ that
\begin{equation}\label{7.24.1}
\begin{split}
\tilde L&=\frac{\rho}{\bb^{-1}\tau} \left(-\frac{\bb^{-1}\tau}{\rho}\bT+a^{-1} \bN\right)+\frac{\rho}{\bb^{-1}\tau}\p_z,\\
\tilde \Lb&=\frac{\rho}{\bb^{-1}\tau} \left(-\frac{\bb^{-1}\tau}{\rho}\bT-a^{-1}\bN \right)-\frac{\rho}{\bb^{-1}\tau}\p_z.
\end{split}
\end{equation}
By using (\ref{fb1}) and (\ref{fb2}) we can write
\begin{align}
&\tilde L=\frac{\rho}{\bb^{-1}\tau}(\fB+\p_z),\qquad \tilde\Lb=\frac{\rho}{\bb^{-1}\tau}(\underline{\fB}- \p_z)\label{barlb}.
\end{align}

In $(\widetilde \M, \tilde \bg)$ we define the following projection tensors
\begin{equation*}
\widetilde {\Pi}^{\tilde\a\tilde\b} = \tilde \bg^{\tilde\a\tilde\b}-\delta_z^{\tilde\a}\delta_z^{\tilde\b},\qquad
\Pi^{\tilde\a \tilde\b} = \tilde \bg^{\tilde\a \tilde\b} + \f12({\tilde L}^{\tilde\a} {\tilde\Lb}^{\tilde\b}
+{\tilde \Lb}^{\tilde\a} {\tilde L}^{\tilde\b}).
\end{equation*}
For the induced metric $\tilde \ga$ on $S_t$, we have $\tilde\ga^{\tilde\a\tilde\b} = \Pi^{\tilde\a\tilde\b}$
if $\tilde\ga$ is regarded as an $S_t$-tangent tensor in $(\widetilde \M, \tilde\bg)$.
We can project $\Pi^{\tilde\a\tilde\b}$ to $(\M,\bg)$ by $\slashed{\Pi}^{\tilde\a\tilde\b}
=\Pi_{\tilde\a'\tilde\b'}\widetilde{\Pi}^{\tilde\a'\tilde\a} \widetilde {\Pi}^{\tilde\b'\tilde\b}$.
Noting that $\widetilde \Pi^{z\tilde\b} =0$ and $\widetilde \Pi^{\a\b}=\bg^{\a\b}$, we have $\slashed{\Pi}^{z\tilde\b}=0$
and
\begin{align}\label{3.16.1.17}
\slashed{\Pi}^{\a\b} =  \bg^{\a \a'} \bg^{\b \b'} \Pi_{ \a' \b'}
= \Pi^{\a\b} = \bg^{\a\b} + \f12 (\tilde L^\a \tilde \Lb^\b + \tilde L^\b \tilde \Lb^\a).
\end{align}
In view of (\ref{barlb}) we have $\tilde L^\a = \frac{\rho}{\bb^{-1}\tau} \fB^\a$ and $\tilde \Lb^\a = \frac{\rho}{\bb^{-1}\tau} \fBb^\a$.
Combining this with the above equation shows that
\begin{align}\label{2.12.3}
\slashed{\Pi}^{\a\b}=\bg^{\a\b}+\f12 \frac{\rho^2}{\bb^{-2}\tau^2}(\fB^\a {\underline{\fB}}^\b+\fB^\b {\underline{\fB}}^\a).
\end{align}
By using (\ref{fb1}) and (\ref{fb2}), we can derive from (\ref{2.12.3}) that
\begin{equation}\label{2.12.4}
\slashed{\Pi}^{\a\b}=\bg^{\a\b}+\bT^\a\bT^\b-\frac{\tir^2}{\bb^{-2}\tau^2}\bN^\a \bN^\b.
\end{equation}
Let $\up{3}\Pi^{\a\b}=\bg^{\a\b}+\bT^\a\bT^\b-\bN^\a \bN^\b$,  the above identity can be recast as
\begin{equation}\label{7.13.1.16}
\slashed{\Pi}^{\a\b}=\up{3}\Pi^{\a\b}+\frac{\rho^2}{\bb^{-2}\tau^2}\bN^\a\bN^\b.
\end{equation}

We now introduce a set of geometric notion on $(\widetilde \M, \tilde \bg)$ which will be used in the
Kirchhoff formula.

\begin{definition}\label{7.21.1}

\begin{enumerate}[leftmargin = 0.8cm]
\item We denote by $\tilde{\pmb{\pi}}$ the second fundamental form of $(\widetilde\Sigma_t, \tilde g)\subset(\widetilde \M, \tilde\bg)$, i.e.
\begin{equation*}
\tilde{\pmb{\pi}}(X, Y)=-\tilde \bg(\widetilde{\bf D}_X {\bf T}, Y)
\end{equation*}
for $X, Y\in \T \widetilde\Sigma_t$ and $\tilde g$ is the induced metric of $\tilde \bg$ on $\bar\Sigma_t$.
We denote by  $\emph{\Tr} \tilde{\pmb{\pi}}$ the trace part of $\tilde{\pmb{\pi}}$.

\item We define the null second fundamental forms on $\widetilde\N^-(p)$ in the extended spacetime $(\widetilde \M, \tilde\bg)$ by
\begin{equation}\label{nf_1}
\tilde{\chi}(X, Y)=\tilde\bg( \widetilde{\bf D}_X \tilde L, Y),\qquad
\tilde{\chib}(X, Y)=\tilde\bg( \widetilde{\bf D}_X \tilde \Lb, Y),
\end{equation}
where $X,\,Y$ are in $\T S_t$. The trace parts of the above symmetric $S_t$-tangent tensor fields are
denoted by $\emph{\tr} \tilde\chi$ and $\emph{\tr} \tilde\chib$.

\item We introduce the connection coefficients
\begin{eqnarray}
\tilde \omega = -\f12 \l \widetilde{\bf D}_{\tilde L} \tilde \Lb, \tilde L\r,
\quad  \underline{\mu}={\tilde L} \emph{\tr} {\tilde\chib}+\f12 \emph{\tr} \tilde \chi \emph{\tr} \tilde\chib,
\quad  \tilde\zb(X) = \f12\tilde\bg(\widetilde {\bf D}_{\tilde L} \tilde \Lb, X),\label{7.15.1}
\end{eqnarray}
where $X\in \T S_t$.
\end{enumerate}
\end{definition}

The following lemma gives some preliminary results on how to represent geometric quantities in $(\widetilde \M, \tilde \bg)$
in terms of geometric quantities in $(\M, \bg)$.

\begin{lemma}\label{2.13.1}
There hold
\begin{align}
&\emph{\tr} \tilde \chi = \frac{\rho}{\bb^{-1}\tau} \emph{\tr} k, \label{area_exp}\\
&\emph{tr} \tilde{\pmb\pi} = \emph{\Tr} \pmb{\pi} - \frac{\bb^2\tir^2}{\tau^2} \pmb{\pi}_{\bN\bN}, \label{chib}\\
&\emph{\tr}\tilde \chib + \emph{\tr} \tilde \chi = 2 \emph{\tr} \tilde{\pmb \pi}, \label{2.13.2}\\
&\tilde{\pmb\pi}_{\tilde N \tilde N} = \frac{\tir^2}{\bb^{-2} \tau^2} \pmb{\pi}_{\bN\bN}, \label{2.14.6} \\
&\tilde \omega  = \omega, \label{7.15.3}\\
& \tilde L\log \bb =\frac{n^{-1}-\bb}{\tau}-\omega. \label{llogb}
\end{align}
where $\tilde{\pmb\pi}$ is defined in Definition \ref{7.21.1} and  $\emph{\tr} \tilde{\pmb\pi}
= -\widetilde {\bf D}_{\ti\mu} {\bf T}_{\ti\nu}\Pi^{\ti\mu\ti\nu}$ is the trace of $\tilde{\pmb \pi}$ restricted to $S_t$.
\end{lemma}

\begin{proof}
To obtain (\ref{area_exp}), we note that
$
\tr \tilde \chi = \Pi^{\tilde\a \tilde \b} \tilde \chi_{\tilde\a \tilde\b}
= \widetilde \bd_{\tilde \a} \tilde L^{\tilde \b} \Pi^{\tilde \a}_{\tilde \b}.
$
In view of (\ref{barlb}) we have
\begin{align*}
\tr \tilde \chi &=\p_{\tilde\a} \left(\frac{\rho}{\bb^{-1} \tau}\right) (\fB^{\tilde \b} +\p_z^{\tilde \b})
\Pi_{\tilde\b}^{\tilde \a} + \frac{\rho}{\bb^{-1}\tau} \widetilde \bd_{\tilde \a} (\fB^{\tilde \b}+\p_z^{\tilde \b})
\Pi_{\tilde \b}^{\tilde \a}.
\end{align*}
From (\ref{chrisplit}) we have $\widetilde \bd_{\tilde \a} \p_z^{\tilde \b}=0$. By the definition of $\Pi_{\tilde \b}^{\tilde \a}$
it is straightforward to check that
$$
(\fB^{\tilde \b} + \p_z^{\tilde \b}) \Pi_{\tilde \b}^{\tilde \a}
= \frac{\bb^{-1}\tau}{\rho} \tilde L^{\tilde \b} \Pi_{\tilde \b}^{\tilde \a} =0.
$$
Therefore
\begin{align*}
\tr \tilde \chi &=\frac{\rho}{\bb^{-1}\tau} \widetilde \bd_{\tilde \a} \fB^{\tilde\b} \Pi_{\tilde\b}^{\tilde\a}
=\frac{\rho}{\bb^{-1}\tau} \left(\widetilde \bd_{\tilde \a} \fB^{\tilde \a} + \f12 \widetilde \bd_{\tilde \a} \fB^{\tilde \b}
(\tilde L^{\tilde \a} \tilde \Lb_{\tilde \b} + \tilde \Lb^{\tilde \a} \tilde L_{\tilde \b})\right).
\end{align*}
By using (\ref{chrisplit}), $\tilde \bg(\fB, \p_z)=0$ and the fact that $\fB$ is independent of $z$, we can derive that
\begin{equation}\label{3.17.2.17}
\widetilde \bd_{\tilde \a} \fB^{\tilde \a} = \bd_\a \fB^\a, \quad \widetilde{\bd}_{\ti\a}\fB^z=0, \quad \widetilde{\bd}_{\a}\fB=\bd_\a \fB.
\end{equation}
We claim
\begin{equation}\label{3.17.1.17}
  \widetilde \bd_{\tilde \a} \fB^{\tilde \b}
(\tilde L^{\tilde \a} \tilde \Lb_{\tilde \b} + \tilde \Lb^{\tilde \a} \tilde L_{\tilde \b}) =0.
\end{equation}
Combining the first identity in (\ref{3.17.2.17}) with (\ref{3.17.1.17}) implies
$$
\tr \tilde \chi = \frac{\rho}{\bb^{-1} \tau} \bd_\a\fB^\a=\frac{\rho}{\bb^{-1}\tau}\tr k.
$$
To see (\ref{3.17.1.17}), we first can obtain $\widetilde \bd_{\ti L}\fB^{\ti\b} =0$ from (\ref{barlb}) and (\ref{7.10.16.4}).
It follows by using the second and the third identities in (\ref{3.17.2.17}), the second identity in (\ref{7.10.16.4}) and (\ref{barlb}) that
\begin{align*}
\widetilde{\bd}_{\ti\a} \fB^{\ti \b} {\ti L}_{\ti\b}&=\widetilde{\bd}_{\ti\a}\fB^z{\ti L}_z+\widetilde{\bd}_{\ti\a} \fB^\b \ti L_\b=\frac{\rho}{\bb^{-1}\tau} {\widetilde\bd}_{\ti\a}\fB^\b \fB_\b=0
\end{align*}
where we used $\bg(\fB, \fB)=-1$. Hence (\ref{3.17.1.17}) is proved and the proof of (\ref{area_exp}) is thus complete.

To show (\ref{chib}), we note that $\tilde \bg(\bT, \p_z) =0$ and $\bT$ is independent of $z$. Thus, by using (\ref{chrisplit})
we can derive that $\widetilde \bd_{\tilde\a} \bT_{\tilde\b} \Pi^{\tilde\a \tilde\b}
= \bd_\a \bT_\b \Pi^{\a\b} .$
Therefore, by using   (\ref{3.16.1.17}) and (\ref{2.12.4}) we deduce that
\begin{align*}
\tr \tilde{\pmb \pi} &=-\bd_\a \bT_\b \slashed{\Pi}^{\a\b} = -\bd_\a \bT_\b \left(\bg^{\a\b}+\bT^\a\bT^\b-\frac{\tir^2}{\bb^{-2}\tau^2} \bN^\a \bN^\b\right)\\
&= \Tr \pmb{\pi} + \frac{\tir^2}{\bb^{-2}\tau^2} \l\bd_\bN \bT, \bN\r
=\Tr \pmb{\pi} - \frac{\tir^2}{\bb^{-2}\tau^2}\pmb{\pi}_{\bN\bN},
\end{align*}
where we used the fact that  $\bg^{\a\b}+\bT^\a\bT^\b$ is the standard  projection to $\Sigma_t\subset \M$.
Hence  (\ref{chib}) is proved.

From the definition of $\tilde \chi$, $\tilde\chib$ and $\tilde{\pmb\pi}$, it is straightforward to derive (\ref{2.13.2}).

Next we derive (\ref{2.14.6}).  Note that $\bT= n^{-1} \p_t$ and $\tilde \bg(\p_t, \tilde N)=0$, we have
$
\tilde{\pmb \pi}_{\tilde N\tilde N} = - n^{-1} \tilde \bg(\widetilde \bd_{\tilde N} \p_t, \tilde N).
$
In view of (\ref{2.13.4}), (\ref{2.13.4.1}) and (\ref{chrisplit}) we can further obtain
\begin{align*}
\tilde{\pmb \pi}_{\tilde N\tilde N} & = -n^{-1} \bar a^2 a^{-1} \left[\tilde \bg(\widetilde \bd_\bN \p_t, \p_z)
+ a^{-1} \tilde \bg(\widetilde \bd_\bN \p_t, \bN)\right] = - n^{-1} \bar a^2 a^{-2} \bg(\bd_\bN \p_t, \bN)\\
& = -\bar a^2 a^{-2} \bg(\bd_\bN \bT, \bN) = \bar a^2 a^{-2} \pmb{\pi}_{\bN\bN} = \frac{\tir^2}{\bb^{-2} \tau^2} \pmb{\pi}_{\bN\bN}.
\end{align*}

To show (\ref{7.15.3}), from (\ref{barlb}), (\ref{7.10.16.4}) and $\l \tilde L, \tilde \Lb\r =-2$ it follows that
\begin{align*}
\tilde \omega &= \f12 \left\l (\fB+\p_z) \left(\frac{\rho}{\bb^{-1}\tau}\right) \tilde L
+\frac{\rho^2}{\bb^{-2}\tau^2}\widetilde\bd_{\fB+\p_z} (\fB+\p_z), \tilde \Lb \right\r
= - (\fB + \p_z) \left(\frac{\rho}{\bb^{-1}\tau}\right).
\end{align*}
Since $\frac{\rho}{\bb^{-1} \tau}$ is independent of $z$, we therefore obtain $\tilde \omega=
- \fB(\frac{\rho}{\bb^{-1}\tau}) = \omega$ which is (\ref{7.15.3}).

To obtain (\ref{llogb}),  by using (\ref{barlb}) and the fact that $\bb$ is independent of $z$,
we can conclude from (\ref{7.12.8.16}) that
\begin{align*} 
\tilde L \log \bb = -\bb \tilde L(\bb^{-1}) = -\frac{\rho}{\bb^{-2}\tau}\fB(\bb^{-1})
= \frac{n^{-1}-\bb}{\tau} - \omega.
\end{align*}
The proof is therefore complete.
\end{proof}

Recall that the Kirchhoff representation formula in \cite{KSob} only holds on the regular part of $\widetilde \N^-(p)$.
To define the regular part, we need the notion of null radius of injectivity.
Since we mainly rely on the time foliation to analyze $\widetilde\N^-(p)$, we only need to introduce the
past null radius of injectivity at $p$ with respect to the global time function $t$.

Let us briefly recall the definition of past null radius of injectivity; one may consult (\cite{KSob,KR2008,KR2, Wang10})
for more details. We parametrize the set
of past null vectors in $\T_p \widetilde \M$ in terms of $\omega \in {\mathbb S}^2$, the standard sphere in ${\mathbb R}^3$.
Then, for each $\omega \in {\mathbb S}^2$, let $\ell_\omega$ be the null vector in $\T_p \widetilde \M$ normalized with
respect to the future, unit, timelike vector $\bT_p$ by
$$
\widetilde \bg(\ell_\omega, \bT_p) = 1
$$
and let $\widetilde \Gamma_\omega(s)$ be the past null geodesic satisfying $\widetilde\Gamma_\omega(0)= p$ and
$\frac{d \widetilde\Gamma_\omega}{ds}(0)= \ell_\omega$. We define the null vector field $\tilde L'$ on $\widetilde \N^-(p)$ by
$$
\tilde L'(\widetilde\Gamma_\omega(s)) = \frac{d}{d s} \widetilde\Gamma_\omega(s)
$$
which may only be smooth almost everywhere on $\widetilde \N^-(p)$ and can be multivalued on a set of exceptional points.
We can choose the parameter $s$ with $s(p)=0$ so that $\widetilde \bd_{\tilde L'} \tilde L'=0$ and $\tilde L'(s) =1$.
This $s$ is called the affine parameter.

We define the past null radius of injectivity $\tilde i_*(p,t)$ at $p$ to be the supremum over all the values $\tau>0$
for which the exponential map
\begin{equation}\label{gp}
{\mathcal G}_p: (t, \omega)\rightarrow \widetilde \Ga_\omega(s(t))
\end{equation}
is a global diffeomorphism from $(t(p)-\tau, t(p))\times {\Bbb S}^2$
to its image in $\widetilde\N^{-}(p)$. We remark that $s$ is a function not
only depending on $t$ but also on $\omega$. We suppress $\omega$
just for convenience. It is known that
$$
\tilde i_*(p,t)=\min\{\tilde s_*(p,t), \tilde \ell_*(p,t)\},
$$
where $\tilde s_*(p, t)$ is defined to be the supremum over all values $\tau>0$ such that the map ${\mathcal G}_p$
is a local diffeomorphism from $(t(p)-\tau, t(p))\times {\Bbb S}^2$ to its image, and $\tilde \ell_*(p,t)$ is
defined to be the smallest value of $\tau>0$ for which there exist two distinct null geodesics $\widetilde \Ga_{\omega_1}(s(t))$ and
$\widetilde \Ga_{\omega_2}(s(t))$ from $p$ which intersect at a point with $t=t_p-\tau.$

We can similarly define in $(\M, \bg)$  the past null radius of injectivity $i_*(p, t)$ with respect to the time foliation.

The following result gives the relation between the causal radius of injectivity in $(\M,\bg)$ and the null radius of
injectivity in $(\widetilde \M, \tilde \bg)$.

\begin{theorem}\label{8.7.2}
For a point $p$ in $(\M, \bg)$, let $c_*(p,t)$ denote the backward causal radius of injectivity at $p$ in $(\M, \bg)$,
which is defined by
$$
c_*(p, t):= \min(\delta_*, i_*(p,t)),
$$
where $\delta_*$ and $i_*(p,t)$ are the past timelike and null radius of injectivity at $p$ in $(\M, \bg)$.
Then there holds
\begin{equation}\label{8.7.1}
c_*(p,t)\le \tilde i_*(p,t).
\end{equation}
\end{theorem}

\begin{proof}
Recall that $\widetilde \N^-(p)$ is composed by the three parts: $\H^+$, $\H^-$ and $\N^-(p)$, and,
according to Lemma \ref{2.27.1.17}, every $\widetilde \Ga_\omega(s(\tau))$ lies  completely in $\H^+$, $\H^-$ or $\N^-(p)$.
Moreover, the proof of Lemma \ref{lem17.1.7} shows that the two sets $\H^\pm$ are  ruled  by the family of curves
\begin{equation}\label{8.11.1}
(\Upsilon_V(\rho(\tau)), \pm\rho(\tau)), \quad V\in {\mathbb H}^2,
\end{equation}
where $\Upsilon_V$ denotes the time-like geodesic in $(\M, \bg)$ with $\Upsilon_V(0)=p$ and $\Upsilon'(0)=V$.
Here the sign of the $z$ coordinate is determined by the fact whether the curves lie in $\H^+$ or $\H^-$.

(i) We first show that for any null geodesic $\widetilde \Gamma_\omega(s)$ lying in $\H^\pm$ there is a unique $V\in {\mathbb H}^2$
such that
\begin{align}\label{eq17.3.12}
\widetilde \Ga_\omega(s(\tau))= \widetilde \Upsilon_V(\tau):= (\Upsilon_V(\rho(\tau)), \pm\rho(\tau))
\end{align}
for $\tau < \min\{\tilde i_*(p, t), c_*(p, t)\}$.

Since the spacetime $(\widetilde\M, \tilde \bg)$ is symmetric about the $z$ coordinate, it suffices to consider
the case that $\widetilde \Ga_\omega$ lies in $\H^+$. We can write
\begin{equation*}
\widetilde\Ga_\omega(s(\tau)):=(\Ga_\omega (s(\tau)), z(s(\tau))).
\end{equation*}
According to Lemma \ref{lem17.1.7}, $z(s(\tau)) = \rho(\tau)$ is the Lorentzian distance of
$\Ga_\omega(s(\tau))$ to $p$ in $(\M, \bg)$ and $\Ga_\omega$ is a timelike geodesic in $(\M, \bg)$ initiating from $p$.
Thus, we can find a $V := (V^0, V^1, V^2) \in {\mathbb H}^2$ such that the geodesic $\Ga_\omega$ can be represented by
$\Upsilon_V$ if parametrized by $\rho$, i.e. $\Ga_\omega(s(\tau)) = \Upsilon_V(\rho(\tau))$. This
shows (\ref{eq17.3.12}). The choice of $V$ is unique, since otherwise there would exist two distance
maximizing time-like geodesics forming a loop at certain $\tau<\delta_*$ in $(\M, \bg)$, which
is impossible by the definition of $\d_*$.

(ii) We now derive the relation between the null geodesic generator of $\widetilde \Ga_\omega(s)$
and the geodesic generator of $\widetilde \Upsilon_V(\rho)$. By using (\ref{eq_2}), (\ref{8.6.1}) and
(\ref{barlb}) we have
\begin{equation*} 
\frac{d}{d\tau} \widetilde\Upsilon_V(\tau) = \frac{d}{d\rho} \widetilde\Upsilon_V\c \frac{d\rho}{d\tau}
=\frac{n\rho}{\bb^{-1}\tau}(\fB+\p_z)\left|_{\widetilde\Upsilon_V(\tau)}\right.
= n \tilde L\left|_{\widetilde\Upsilon_V(\tau)}\right.
= n\tilde L\left|_{\widetilde\Ga_\omega (s(\tau))}\right..
\end{equation*}
On the other hand, since $\tilde L'(s)=1$ implies $\frac{d s}{d \tau} = \aaa n$ along $\widetilde \Ga_\omega(s(\tau))$
with $\aaa^{-1} = \l \tilde L', \bT\r$, we may use $\ti \Upsilon_V(\tau)=\widetilde\Ga_\omega (s(\tau))$ to derive that
$$
\frac{d}{d\tau} \widetilde\Upsilon_V(\tau) = \frac{d}{d \tau} \widetilde \Ga_\omega(s(\tau))
= \frac{d}{d s} \widetilde\Ga_\omega(s(\tau))\c \frac{d s}{d \tau}
= \aaa n \tilde L'\left|_{\widetilde\Ga_\omega(s(\tau))}\right.
$$
Combining the above two equations we can obtain $\tilde L' = \aaa^{-1} \tilde L$.  Using $\l \ell_\omega, \bT_p\r=1$
we can see that $\aaa \rightarrow 1$ as $q$ approaches $p$ along $\widetilde \Ga_\omega$.

By the definition of $V^0$ and (\ref{eq_1}), along the timelike geodesic $\Upsilon_V$ there holds
\begin{equation*}
\lim_{t\rightarrow t_p}\frac{\rho}{\bb^{-1}(t_p-t)}= \frac{1}{V^0}.
\end{equation*}
This together with $\widetilde \bd_{\tilde L'} \tilde L' =0$, $\l \ell_\omega, \bT_p\r =1$ and (\ref{eq_1}) imply that
\begin{equation}\label{8.6.3}
\tilde L' = \frac{1}{V^0} (\fB+\p_z), \quad \aaa= \frac{\rho}{\bb^{-1}\tau} V^0.
\end{equation}

(iii) We now show that if $\tilde \ell_*(p,t)\ge \tilde s_*(p,t)$ then $\tilde s_*(p, t) \ge c_*(p, t)$.
To this end, consider an arbitrary null geodesic $\widetilde \Ga_\omega$. If $\widetilde \Ga_\omega(s(\tau))$
is contained in $\N^-(p)$, by the definition of $c_*(p, t)$, $\widetilde \Ga_\omega(s(\tau))$ does not
contain any null conjugate point on $(0, c_*(p, t))$. So it needs only to consider the case that $\widetilde \Ga_\omega$
is contained in $\H^+$ or $\H^-$. By symmetry, it suffices to consider the case that $\widetilde \Ga_\omega$ is
contained in $\H^+$. According to \cite{Choquet,HE}, it suffices to show that $\tr \tilde \chi'>-\infty$
for $\tau:=t_p-t<c_*(p, t)$, where $\tilde \chi'$ denotes the null second fundamental form defined by (\ref{nf_1})
with $\tilde L$ replaced by $\tilde L'$ and $\tr \tilde \chi'$ denotes its trace. By using (\ref{barlb}) and (\ref{8.6.3}),
it is straightforward to obtain
\begin{equation*}
\tr\tilde\chi = \frac{\rho}{\bb^{-1}\tau} V^0 \tr\tilde\chi'.
\end{equation*}
This together with (\ref{area_exp}) shows that
\begin{equation}\label{8.7.5}
V^0 \tr \tilde \chi' = \tr k.
\end{equation}
Note that the timelike geodesic $\Upsilon_V$ in $(\M, \bg)$ from $p$ reaches a conjugate point $q$ iff $\tr k\rightarrow -\infty$
as points approach $q$ along this geodesic. Since $\tilde \ell_*(p, t)\ge \tilde s_*(p, t)$, we can conclude from (\ref{8.7.5}) that,
$\tr \tilde \chi'$ does not diverge to $-\infty$ along $\widetilde \Ga_\omega(s(\tau))$ iff $\tr k$ does not along
$\Upsilon_V(\rho(\tau))$. Since $\tr k >-\infty$ along $\Upsilon_V(\rho(\tau))$ on $(0, c_*(p, t))$, we must have
$\tr \tilde\chi'>-\infty$ along $\widetilde \Ga_\omega(s(\tau))$ on $(0, c_*(p, t))$. Therefore $\tilde s_*(p, t) \ge c_*(p, t)$.

(iv) Finally we show that $\tilde i_*(p, t) \ge c_*(p, t)$. Suppose this is not true, i.e. $\tilde i_*(p, t) < c_*(p, t)$,
we will derive a contradiction. By using the claim in (iii), we must have $\tilde \ell_*(p,t)< \tilde s_*(p,t)$.
Thus there exist two distinct null geodesics $\widetilde \Ga_{\omega_1}(s(\tau))$ and $\widetilde \Ga_{\omega_2}(s(\tau))$ intersecting at
some point $q$ with  $\tau= \tilde \ell_*(p,t)$. If $q \in \N^-(p)$, then Lemma \ref{2.27.1.17} implies that
$\widetilde \Ga_{\omega_1}(s(\tau))$ and $\widetilde \Ga_{\omega_2}(s(\tau))$ are both contained in $\N^-(p)$ and intersect at $q$;
this can not happen since $\tau < c_*(p, t)$. We may assume $q \in \H^+ \cup \H^-$. By symmetry we only need to consider the case that
$q \in \H^+$. Now by Lemma \ref{2.27.1.17} both $\widetilde \Ga_{\omega_1}(s(\tau))$ and $\widetilde \Ga_{\omega_2}(s(\tau))$
are contained in $\H^+$.  According to (i), we can find two distinct vectors $V_1, V_2 \in {\mathbb H}^2$ such that
\begin{equation*}
\widetilde \Ga_{\omega_i}(s(\tau))=\widetilde \Upsilon_{V_i}(\tau):= (\Upsilon_{V_i}(\rho(\tau)), \rho(\tau)), \quad i =1, 2
\end{equation*}
for $\tau< \tilde \ell_*(p,t)$. Thus there exists a point $q'\in \I^-(p)$ such that $\Upsilon_{V_1}(\rho(\tau))
=\Upsilon_{V_2}(\rho(\tau)) = q'$ when $\tau= \tilde \ell_*(p,t)<c_*(p,t)$. This is impossible by the definition of $c_*(p, t)$
and in particular the definition of $\d_*$. Therefore $\tilde i_*(p,t)\ge c_*(p,t)$ and the proof is complete.
\end{proof}

According to Theorem \ref{8.7.2}, for any $\tau_0<c_*(p, t)$, $\widetilde \N^-(p)\cap \{t_p-t<\tau_0\}$ is a regular part of
$\widetilde \N^-(p)$ on which the Kirchhoff formula in \cite{KSob,Shao,Wang10} holds. By adapting the version of the Kirchhoff
formula in \cite{Shao,Wang10} with the null frame $\{\tilde L, \tilde \Lb, e_C, C=1,2\}$, where $\{e_C\}_{C=1,2}$ is an
orthonormal frame on $S_t$, we have the following result which will be used to proved Theorem \ref{7.13.2.16}.

\begin{proposition}\label{7.15.8}
Let $p$ be any point in $\M$, let $t_0$ verify $0<\tau_0 := t_p-t_0<c_*(p, t)$, and set $\H_* := \widetilde \N^-(p)\cap \{t_p-t\le \tau_0\}$.
Let $\phi_I$ be an $\M$-tangent tensor on $(\widetilde \M, \tilde \bg)$. For any $\M$-tangent tensor $J_I$ at $p$ of
the same type as $\phi_I$, let $A^I$ be an $\M$-tangent tensor field satisfying
\begin{equation}\label{tsa}
\widetilde \bd_{\tilde L} A_I + \left(\frac{1}{2} \tr \tilde\chi +\frac{n^{-1}-\bb}{\tau}\right) A_I=0,
\quad \lim_{t\rightarrow t_p} \tau A_I=J_I.
\end{equation}
along $\H_*$. Then there holds
the Kirchhoff formula
\begin{align}\label{2.14.2}
4\pi (n \l \phi, J\r)(p)
&= -\int_{S_{t_0}} \left(\widetilde{\bf D}_{\tilde\Lb} \phi_I + \f12 \emph{\tr} \tilde\chib \phi_I\right) A^I d{\mu_{S_{t_0}}} \nn\\
& \quad \, -\int_{\H_*} \Box_{\tilde\bg} \phi_I A^I \bb d\mu + \Er[\phi]
\end{align}
with
\begin{align}
\Er[\phi] &= \int_{\H_*} \left(2\tilde\zb^C \widetilde{\bf D}_C \phi_I+\widetilde \sD \phi_I \right) A^I\bb d\mu\nn\\
& \quad \, +\f12\int_{\H_*}\left(-(\widetilde{\bf R}*\phi)_I
+\left(\underline{\mu}-\omega \emph{\tr}\tilde \chib\right)\phi_I \right) A^I \bb d\mu, \label{2.14.1}
\end{align}
where $d\mu=n d\mu_{S_t} dt$ on $\H_*$, $\underline\mu$ and $\tilde\zb$ are defined in (\ref{7.15.1}),
and for $I = \{\mu_1, \cdots, \mu_l\}$,
\begin{align}\label{2017.1.1.1}
(\widetilde {\bf R} * \phi)_I = \sum_{i=1}^l \tensor{\widetilde {\bf R}}{_{\mu_i}^\a _{\tilde \Lb\tilde L}}
\tensor{\phi}{_{\mu_1\cdots\mu_{i-1} \a \mu_{i+1}\cdots \mu_l}}.
\end{align}
\end{proposition}

 In Section \ref{7.15.7}, we will give an alternative proof for the Kirchhoff formula in  Proposition \ref{7.15.8}.  In the sequel we will prove
Theorem \ref{7.13.2.16} by using Proposition \ref{7.15.8}.

\subsection{Proof of Theorem \ref{7.13.2.16}}

We first give the following result which represents $\tr\tilde\chib$, $\tilde\zb$, $\widetilde \bR$ and
$\underline{\mu}-\omega \tr \tilde \chib$ in terms of the quantities in $(\M, \bg)$.

\begin{lemma}\label{7.18.2}
\begin{align}
\tilde\zb^\b & = -\frac{\tir}{\bb^{-1}\tau} \zb^\sl{A}e_{\sl{A}}^\b
   - \frac{\rho^2}{\bb^{-1}\tau \tir} \omega \bN^\b, \label{torsion}\\
\emph{\tr}\tilde{\chib} & =-\frac{\rho}{\bb^{-1} \tau} \emph{\tr} k+2\left(\emph{\Tr} \pmb{\pi}
-\frac{\tir^2}{\bb^{-2}\tau^2} \pmb{\pi}_{\bN\bN}\right),\label{2.14.4} \\
\underline{\mu}-\omega \emph{\tr} \tilde\chib
& = 2\frac{\rho}{\bb^{-1}\tau}\omega \emph{\tr} k +\frac{\rho^2}{\bb^{-2}\tau^2}(\bR_{\fB\fB}+|\hk|^2) \nonumber\\
& \quad \, +2 \frac{\rho}{\bb^{-1}\tau} \left(\fB+\frac{\emph{\tr} k}{2}-\frac{\bb^{-1}\tau\omega}{\rho}\right)
\left(\emph{\Tr} \pmb{\pi}-\frac{\tir^2}{\bb^{-2}\tau^2} \pmb{\pi}_{\bN\bN}\right). \label{7.15.2} \\
\f12 \tensor{\widetilde {\bf R}}{_\b^\a _{\tilde L} _{\tilde\Lb}}
& =\frac{\tir}{\bb^{-1}\tau}\tensor{{\bf R}}{_\b^\a _\bT _\bN}.
\label{7.18.1}
\end{align}
\end{lemma}

\begin{proof}
We first derive (\ref{torsion}). Note that $\tilde \zb^\beta = \f12 (\widetilde\bd_{\tilde L} \tilde \Lb)^{\tilde \a} \Pi_{\tilde \a}^\b$.
By using (\ref{barlb}) and $\tilde \Lb^{\tilde \a} \Pi_{\tilde \a}^\b=0$ we have
\begin{align*}
\tilde \zb^\beta = \f12\frac{\rho}{\bb^{-1}\tau} \widetilde \bd_{\tilde L} (\fBb - \p_z)^{\tilde \a} \Pi_{\tilde \a}^\b
= \f12\frac{\rho^2}{\bb^{-2} \tau^2} \widetilde\bd_{\fB+\p_z} (\fBb-\p_z)^{\tilde \a} \Pi_{\tilde \a}^\b.
\end{align*}
In view of (\ref{chrisplit}), (\ref{7.10.16.4}) and the fact that $\fBb$ is independent of $z$, we then obtain
$$
\tilde \zb^\b = \f12\frac{\rho^2}{\bb^{-2}\tau^2} \widetilde\bd_\fB \fBb^{\tilde\a} \Pi_{\tilde\a}^\b
= \f12\frac{\rho^2}{\bb^{-2}\tau^2} \bd_\fB \fBb^\a \sl{\Pi}_\a^\b
$$
where we also used  $\Pi_\a^\b = \slashed{\Pi}_\a^\b$ due to (\ref{3.16.1.17}).
In view of the above identity  and (\ref{7.13.1.16}) we can see that
\begin{equation}\label{7.12.10.16}
\tilde\zb^\b =\f12 \frac{\rho^2}{\bb^{-2}\tau^2}\left(\l \bd_\fB \fBb, e_\sl{A}\r e_\sl{A}^\b
+\frac{\rho^2}{\bb^{-2}\tau^2}\l\bd_\fB \fBb, \bN\r \bN^\b\right).
\end{equation}
By using (\ref{fb2}) we have
\begin{align}\label{7.11.4.16}
\bd_\fB \fBb &= - \bd_\fB \left(\frac{\bb^{-1}\tau}{\rho}\bT+\frac{\tir}{\rho}\bN \right)\nn\\
& = -\fB \left(\frac{\bb^{-1}\tau}{\rho}\right)\bT - \fB\left(\frac{\tir}{\rho}\right) \bN
 - \frac{\bb^{-1}\tau}{\rho} \bd_\fB \bT - \frac{\tir}{\rho} \bd_\fB \bN .
\end{align}
Since $\l \bT, e_\sl{A}\r = \l \bN, e_\sl{A}\r =0$, it follows that
\begin{equation}\label{7.11.3.16}
\l\bd_\fB \fBb, e_\sl{A}\r = -\frac{\bb^{-1}\tau}{\rho}\l \bd_\fB \bT, e_\sl{A}\r
- \frac{\tir}{\rho}\l \bd_\fB \bN, e_\sl{A}\r.
\end{equation}
From  (\ref{7.10.16.3}) we have $\l \bd_\fB \bT, e_\sl{A}\r = \frac{\tir}{\rho} \zb_\sl{A}$. Recall from \cite[page 17]{Wang16}
that $\sn_\sl{A} \log a=-\l\bd_\bN \bN, e_\sl{A}\r$. This together with (\ref{fb1}) and (\ref{8.3.5}) gives
\begin{align*}
\l \bd_\fB \bN, e_\sl{A}\r &= - \frac{\bb^{-1}\tau}{\rho} \l \bd_\bT \bN, e_\sl{A}\r + \frac{\tir}{\rho} \l \bd_\bN \bN, e_\sl{A}\r \\
& = -\frac{\bb^{-1}\tau}{\rho} \left(k_{\Nb \sl{A}} + \frac{\bb^{-1}\tau}{\tir} \l \bd_\bT\bT, e_\sl{A}\r\right) - \frac{\tir}{\rho}\sn_\sl{A}\log a.
\end{align*}
By using  (\ref{8.3.3}) and (\ref{7.10.16.3}), we then obtain
\begin{align*}
\l \bd_\fB \bN, e_\sl{A}\r
&= -\frac{\bb^{-1}\tau}{\rho}\left(k_{\Nb \sl{A}} + \frac{\bb^{-1}\tau}{\tir} \l \bd_\bT\bT, e_\sl{A}\r\right)
-\frac{\tir}{\rho} \frac{\bb^{-1} \tau}{\tir} \left(\pi_{\bN \sl{A}} - k_{\Nb \sl{A}}\right) \\
& = \frac{\bb^{-1}\tau}{\rho}\left(- \frac{\bb^{-1}\tau}{\tir} \l \bd_\bT\bT, e_\sl{A}\r + \l \bd_\bN \bT, e_\sl{A}\r \right) \\
& = \frac{\bb^{-1}\tau}{\tir} \l \bd_\fB \bT, e_\sl{A}\r = \frac{\bb^{-1} \tau}{\rho} \zb_\sl{A}.
\end{align*}
Combining the above results with (\ref{7.11.3.16}), we obtain
\begin{align}\label{7.11.3.16.1}
\l\bd_\fB \fBb, e_\sl{A}\r = - 2\frac{\bb^{-1} \tau \tir}{\rho^2} \zb_\sl{A}.
\end{align}

Next we calculate $\l \bd_\fB \fBb, \bN\r$. It follows from (\ref{7.11.4.16}) that
\begin{equation}\label{7.12.9.16}
\l\bd_\fB \fBb, \bN\r = - \fB\left(\frac{\tir}{\rho}\right) - \frac{\bb^{-1} \tau}{\rho} \l \bd_\fB \bT, \bN\r.
\end{equation}
We have $\fB(\tir/\rho) = \fB(\tir)/\rho - \tir/\rho^2$. From $\tir^2 = \bb^{-2} \tau^2 - \rho^2$ it follows that
$2 \tir \fB(\tir) = 2 \bb^{-1} \tau \fB(\bb^{-1} \tau) - 2\rho$. Thus
$$
\fB(\tir) = \frac{\bb^{-1} \tau}{\tir} \fB(\bb^{-1} \tau) - \frac{\rho}{\tir}.
$$
Consequently
$$
\fB\left(\frac{\tir}{\rho}\right) = \frac{\bb^{-1}\tau}{\rho \tir} \fB(\bb^{-1}\tau) - \frac{1}{\tir} -\frac{\tir}{\rho^2}
= \frac{\bb^{-1} \tau}{\rho \tir} \fB(\bb^{-1} \tau) - \frac{\bb^{-2} \tau^2}{\rho^2 \tir}.
$$
In view of (\ref{eq_2}), we have $\fB(\tau) = \frac{n^{-1}\bb^{-1} \tau}{\rho}$.
This together with (\ref{7.12.8.16}) shows that
$$
\fB(\bb^{-1} \tau) = \bb^{-1} \fB(\tau) + \tau \fB(\bb^{-1}) = \frac{\bb^{-1}\tau}{\rho} + \frac{\bb^{-2}\tau^2}{\rho} \omega.
$$
Therefore
$$
\fB\left(\frac{\tir}{\rho}\right) = \frac{\bb^{-3}\tau^3}{\rho^2 \tir} \omega.
$$
Combining this with (\ref{7.12.9.16}) and using (\ref{7.15.3}) and (\ref{7.15.6}) we obtain
\begin{equation}\label{7.12.11.16}
\l\bd_\fB \fBb, \bN\r = - \frac{\bb^{-3}\tau^3}{\rho^2 \tir} \omega - \frac{\bb^{-1} \tau}{\rho} \l \bd_\fB \bT, \bN\r
= - 2 \frac{\bb^{-3}\tau^3}{\rho^2 \tir} \omega.
\end{equation}
The combination of (\ref{7.12.10.16}), (\ref{7.11.3.16.1})  and (\ref{7.12.11.16}) then shows (\ref{torsion}).

The equation (\ref{2.14.4}) follows directly from (\ref{area_exp}), (\ref{chib}) and (\ref{2.13.2}) in Lemma \ref{2.13.1}.

Next we derive (\ref{7.15.2}). Note that $\tr \tilde\chib$ is independent of $z$. It follows from (\ref{barlb}),
(\ref{area_exp}), the definition of $\underline{\mu}$ from (\ref{7.15.1}) and (\ref{2.14.4}),
\begin{align}\label{7.15.5}
\underline{\mu} &= \frac{\rho}{\bb^{-1}\tau} \left(\fB+ \frac{1}{2} \tr k\right) \tr \tilde \chib \nn\\
& = \frac{\rho}{\bb^{-1}\tau} \left(\fB+\f12\tr k\right)
\left(-\frac{\rho}{\bb^{-1}\tau} \tr k +2\left(\Tr \pmb{\pi}-\frac{\tir^2}{\bb^{-2}\tau^2}\pmb{\pi}_{\bN\bN}\right)\right).
\end{align}
In view of Definition \ref{7.21.1.0} (3),  we have
\begin{align*}
\left(\fB+\f12\tr k\right) \left(-\frac{\rho}{\bb^{-1}\tau} \tr k \right)
& = \omega \tr k-\frac{\rho}{\bb^{-1}\tau} \left(\fB \tr k+\f12 (\tr k)^2\right).
\end{align*}
By using the Raychaudhouri equation \cite[Section 3]{Wang16}
\begin{equation}\label{7.15.4}
\fB \tr k+\f12 (\tr k)^2 =-\bR_{\fB \fB}-|\hk|^2,
\end{equation}
we further obtain
\begin{align*}
\left(\fB+\f12\tr k\right) \left(-\frac{\rho}{\bb^{-1}\tau} \tr k \right)
= \tr k \omega+\frac{\rho}{\bb^{-1}\tau} (\bR_{\fB\fB}+|\hk|^2).
\end{align*}
Therefore, it follows from (\ref{7.15.5}) that
\begin{align*}
\underline{\mu} = \frac{\rho}{\bb^{-1}\tau} \left[\omega \tr k + \frac{\rho}{\bb^{-1} \tau} (\bR_{\fB\fB} + |\hk|^2)
+ 2 \left(\fB + \f12 \tr k\right) \left(\Tr \pmb{\pi} - \frac{\tir^2}{\bb^{-2} \tau^2} \pmb{\pi}_{\bN\bN}\right)\right].
\end{align*}
This together with  (\ref{2.14.4}) shows (\ref{7.15.2}).

Finally we prove (\ref{7.18.1}).  From the identity
\begin{align*}
\tensor{\widetilde \bR}{_{\tilde \b}^{\tilde\a} _{\tilde\ga \tilde\d}}
&=\p_{\tilde \d} {\tilde \Ga}^{\tilde\a}_{\tilde \b \tilde \ga} - \p_{\tilde \ga} {\tilde \Ga}^{\tilde \a}_{\tilde \b \tilde \d}
+{\tilde\Ga}^{\tilde \a}_{\tilde \d \tilde \eta}{\tilde \Ga}^{\tilde \eta}_{\tilde \b \tilde \ga}
-{\tilde \Ga}_{\tilde \ga \tilde \eta}^{\tilde \a} {\tilde \Ga}^{\tilde \eta}_{\tilde \b \tilde \d}
\end{align*}
and (\ref{chrisplit}) it follows that
\begin{equation*} 
\tensor{\widetilde \bR}{_\b^\a_{z\ga}}=0 \quad \mbox{ and } \quad
\tensor{\widetilde\bR}{_\b^\a_{\ga\delta}}=\tensor{\bR}{_\b^\a_{\ga\delta}}.
\end{equation*}
Therefore, by using  (\ref{fb1}), (\ref{fb2}), (\ref{barlb}) we can obtain (\ref{7.18.1}).
\end{proof}

\begin{proof}[Proof of Theorem \ref{7.13.2.16}]
For $A_I$ satisfying (\ref{tsa2}), by viewing it as an $\M$-tangent tensor in $\widetilde \M$, we may use (\ref{barlb}),
(\ref{area_exp}) and (\ref{chrisplit}) to verify that $A_I$ satisfies (\ref{tsa}) along $\H_*$. Thus, we may
represent $\phi_I$ by the formula (\ref{2.14.2}) in Proposition \ref{7.15.8}. Since $\phi_I$ is $\M$-tangent, we have
from (\ref{ael}), (\ref{barlb}) and (\ref{2.14.4}) that
\begin{align*}
\int_{\H_*} \Box_{\tilde \bg} \phi_I A^I \bb d\mu
& = \int_{t_p-t_0}^{t_p} \int_{S_t} \Box_{\bg} \phi_I A^I \bb n d\mu_{S_t} d t
= 2 \int_{t_p-t_0}^{t_p} \int_{\Sigma_t} \Box_{\bg} \phi_I A^I \frac{\tau}{\rho} n d \mu_{\Sigma_t} dt\\
& = 2 \int_{\I_*^-(p)} F_I A^I \frac{\tau}{\rho} n d \mu_{\Sigma_t} dt
\end{align*}
and
\begin{align*}
& \int_{S_{t_0}} \left(\widetilde \bd_{\tilde \Lb} \phi_I + \frac{1}{2} \tr \tilde \chib \phi_I\right) A^I d\mu_{S_{t_0}}\\
& = \int_{S_{t_0}} \left[\frac{\rho}{\bb^{-1}\tau} \bd_{\fBb} \phi_I - \f12 \frac{\rho}{\bb^{-1}\tau} \tr k \phi_I
+ \left(\Tr \pmb\pi - \frac{\tir^2}{\bb^{-2}\tau^2} \pmb\pi_{\bN\bN}\right)\phi_I\right] A^I d\mu_{S_{t_0}}\\
& = 2\int_{\Sigma_{t_0}} \left[\bd_{\fBb} \phi_I - \f12 \tr k \phi_I
+ \frac{\bb^{-1}\tau}{\rho} \left(\Tr \pmb\pi - \frac{\tir^2}{\bb^{-2}\tau^2} \pmb\pi_{\bN\bN}\right)\phi_I\right] A^I d\mu_{\Sigma_{t_0}}.
\end{align*}
Therefore
\begin{align}\label{17.1.14}
4 \pi (n \bg(\phi, J))(p) = -2 \int_{\I_*^-(p)} F_I A^I \frac{\tau}{\rho} n d\mu_{\Sigma_t} dt + 2 \I_1 + Er[\phi],
\end{align}
where $\I_1$ is given in Theorem \ref{7.13.2.16} and $Er[\phi]$ is given by (\ref{2.14.1}) in Proposition \ref{7.15.8}.

We need to consider $Er[\phi]$. By the divergence theorem, (\ref{ael}) and  (\ref{7.13.1.16}), we have
\begin{align*}
\int_{S_t} \widetilde \sD \phi_I A^I \bb d\mu_{S_t}
&= - \int_{S_t} \widetilde \sn \phi_I \widetilde \sn(A^I \bb) d\mu_{S_t} \\
& = -2 \int_{\Sigma_t} \left[\sn \phi_I \sn (A^I \bb) + \frac{\rho^2}{\bb^{-2} \tau^2} \bd_{\bN} (\bb A^I) \bd_{\bN} \phi_I \right]
\frac{\bb^{-1}\tau}{\rho} n d\mu_{\Sigma_t}.
\end{align*}
In view of (\ref{torsion}) we then obtain
$$
\int_{\H_*} \left(2\tilde \zb^C \ti D_C \phi_I + \widetilde \sD \phi_I\right) A^I \bb d\mu = 2 \I_2
$$
with $\I_2$ defined in Theorem \ref{7.13.2.16}. In view of (\ref{7.15.2}), (\ref{7.18.2}) and (\ref{ael}) we can see that
$$
 \f12\int_{\H_*} \left(-(\widetilde \bR*\phi)_I + (\underline{\mu}-\omega \tr \ti\chib) \phi_I \right) A^I \bb d\mu = 2 \I_3,
$$
where $\I_3$ is defined in Theorem \ref{7.13.2.16}. Consequently $Er[\phi] = 2 \I_2 + 2 \I_3$. Combining this with
(\ref{17.1.14}) we therefore complete the proof.
\end{proof}

\section{\bf Appendix. Proof of Proposition \ref{7.15.8}}\label{7.15.7}
\setcounter{equation}{0}


In this section we give an alternative proof of the Kirchhorff formula in Proposition \ref{7.15.8} by using a multiplier type approach.
As in \cite{Wang10}, we can decompose $\Box_{\tilde \bg} \phi_I$ as
\begin{align}\label{f1}
\Box_{\tilde\bg} \phi_I &= -\widetilde\bd_{\tilde L}( \widetilde\bd_{\tilde\Lb} \phi)_I
+2 \tilde\zb^C \widetilde{\bd}_C \phi_I - \f12 \tr\tilde\chib\widetilde\bd_{\tilde L} \phi_I
-\f12(\tr\tilde\chi-2 \omega) \widetilde\bd_{\tilde \Lb} \phi_I \nn\\
&\quad \,  + \delta^{AB} {\widetilde \sn}_A {\tilde\sn}_B \phi_I - \f12 (\widetilde\bR *\phi)_I,
\end{align}
where $(\tilde {\bf R}*\phi)_I$ is defined by (\ref{2017.1.1.1}).

Let $v=v_t$ and $\psi_I=v^{\f12} \phi_I$. Note that $\tilde L v = \tr\tilde\chi v$ and $\tilde \Lb v = \tr\tilde\chib v$.
We can derive that
\begin{align*}
\widetilde \bd_{\tilde L}( \widetilde\bd_{\tilde \Lb} \psi)_I
& = \widetilde\bd_{\tilde L}\left(v^{\f12} \widetilde\bd_{\tilde \Lb} \phi +\f12 v^\f12 \tr \tilde\chib \phi\right)_I \\
& = v^{\f12}\widetilde\bd_{\tilde L}(\widetilde\bd_{\tilde \Lb} \phi)_I+\f12 v^{\f12} \left(\tr \tilde\chi \widetilde\bd_{\tilde\Lb} \phi_I
+ \tr \tilde\chib \widetilde\bd_{\tilde L}\phi_I\right) + \f12 \underline{\mu} \psi_I.
\end{align*}
By contracting (\ref{f1}) with $\bb v A$ and using the above identity it follows that
\begin{align}
\Box_{\tilde\bg} \phi_I A^I \bb v & = -\left(\widetilde\bd_{\tilde L} (\widetilde\bd_{\tilde \Lb} \psi)_I-\f12 \underline{\mu} \psi_I\right)
A^I v^\f12\bb+\left(2\tilde \zb^C {\widetilde \bd}_C \phi_I + \omega \widetilde\bd_{\tilde\Lb} \phi_I\right) A^I v\bb \nn\\
&\quad \, +\widetilde{\sD} \phi_I A^I v\bb - \f12( \widetilde \bR*\phi)_I A^I v \bb. \label{2.16.3}
\end{align}
Note that the first term on the right hand side can be recast as
\begin{align*}
-\widetilde \bd_{\tilde L}(\widetilde\bd_{\tilde\Lb}  \psi)_I A^I v^{\f12}\bb
& = -\tilde L\left(\bb \widetilde \bd_{\tilde\Lb} \psi_I A^I v^\f12\right)
+ \widetilde \bd_{\tilde L} ( v^\f12\bb A)^I \widetilde \bd_{\tilde \Lb} \psi_I\\
& = -\tilde L\left(\bb \widetilde \bd_{\tilde \Lb} \psi_I A^I v^\f12\right)
+\bb \left(A^I v^{\f12}\tilde L\log \bb +\widetilde \bd_{\tilde L}(v^\f12 A)^I \right)\widetilde \bd_{\tilde \Lb} \psi_I\\
&=-\tilde L\left(\bb \widetilde\bd_{\tilde\Lb} \psi_I A^I v^\f12\right)\\
& \quad \, +\bb\left( \left(\frac{n^{-1}-\bb}{\tau}-\omega\right) A^I+ \left(\widetilde \bd_{\tilde L} A^I
+\f12 \tr\tilde \chi A^I \right) \right) v^\f12 \widetilde \bd_{\tilde \Lb} \psi_I,
\end{align*}
where for the last equality we employed (\ref{llogb}). By using (\ref{tsa}) and $\tilde \Lb v=\tr\tilde \chib v$,
we have
\begin{align}\label{2.16.2}
-\widetilde \bd_{\tilde L}(\widetilde \bd_{\tilde \Lb} \psi)_I A^I v^{\f12}\bb
= - \tilde L\left(\bb \widetilde \bd_{\tilde \Lb} \psi_I A^I v^\f12\right)
-\bb\omega\left(\widetilde \bd_{\tilde \Lb} \phi_I+\f12 \tr\tilde \chib\phi_I\right) A^I v.
\end{align}
Substituting (\ref{2.16.2}) into (\ref{2.16.3}) and integrating along $\H_*$, we have, in view of $\frac{d}{d\tau}=n\tilde L$, that
\begin{align}
\int_{\H_*} \Box_{\tilde\bg} \phi_I A^I \bb v d\mu
=-\int_{S_{t_0}} \bb\widetilde\bd_{\tilde\Lb} \psi_I A^I v^{-\f12} d\mu_{S_{t_0}}
+\lim_{t\rightarrow t_p}\int_{S_t} \bb{\widetilde \bd}_{\tilde \Lb} \psi_I A^I v^{-\f12}  d\mu_{S_{t}}
+\Er[\phi], \nn
\end{align}
where  $\Er[\phi]$ is defined in Proposition \ref{7.15.8} and $d \mu=n d\mu_{S_{t}} dt$.
We claim that
\begin{equation}\label{8.3.1}
\lim_{t\rightarrow t_p} \int_{S_t} {\widetilde \bd}_{\tilde \Lb} \psi_I A^I v^{-\f12}  d\mu_{S_{t}}
= -4\pi( n\l \phi, J\r)(p).
\end{equation}
Using this claim we can conclude that
\begin{align*}
\int_{\H_*} \Box_{\tilde\bg} \phi_I A^I d\mu = -\int_{S_{t_0}} \bb\widetilde\bd_{\tilde\Lb} \psi_I A^I v^{-\f12} d\mu_{S_{t_0}}
-4\pi (n \l \phi, J\r)(p)+ \Er[\phi]
\end{align*}
which is the desired formula.

To obtain (\ref{8.3.1}), we note that
\begin{equation}\label{8.5.1}
\widetilde\bd_{\tilde \Lb} \psi_I A^I v^{-\f12}
= \left(\f12 \tr\tilde\chib\phi_I+\widetilde\bd_{\tilde\Lb} \phi_I\right)A^I.
\end{equation}
Recall that $\tau = t_p -t$. By using (\ref{2.14.4}) and the boundedness of $\Tr\pi$ and $\pi_{\bN\bN}$
near $p$ we have $\tau \tr \tilde \chib + \bb \rho \tr k \rightarrow 0$ as $t\rightarrow t_p$. Recall also  that
$$
\rho \tr k \rightarrow 2, \quad \bb \rightarrow \frac{1}{n(p)} \quad \mbox{and} \quad
\frac{|S_t|}{4\pi n^2\tau^2} \rightarrow 1 \quad \mbox{ as }  t\rightarrow t_p
$$
which have been proved in \cite{Eric,Wang16, WangCMC}. We therefore obtain $\tau \tr \tilde \chib \rightarrow
-2/n(p)$ and hence, by using (\ref{8.5.1}) and $\lim_{t\rightarrow t_p}\tau A_I=J_I$ in (\ref{tsa}), we can obtain (\ref{8.3.1}).


\begin{thebibliography}{9999}

\bibitem{Choquet} Choquet-Bruhat, Y., {\it General relativity and the Einstein equations}, Oxford Mathematical Monographs.
Oxford University Press, Oxford, 2009.

\bibitem{chsh}  Chrusciel, P. and  Shatah, J.,
{\it  Global existence of solutions of the Yang-Mills equations on globally hyperbolic four-dimensional Lorentzian manifolds},
Asian J. Math. 1 (1997), no. 3, 530--548.

\bibitem{EM2}  Eardley, D.,  Moncrief, V.,
{\it The global existence of Yang-Mills-Higgs fields in 4-dimensional Minkowski space II. Completion of proof},
Comm. Math. Phys. 83 (1982), no. 2, 193–-212.

\bibitem{Evans} Evans, L.  {\it Partial differential equations.} Second edition. Graduate Studies in
Mathematics, 19. American Mathematical Society, Providence, RI, 2010. xxii+749 pp.

\bibitem{Fried} Friedlander, H. G.,
{\it The Wave Equation on a Curved Space-time}, Cambridge University Press, 1976.

\bibitem{HE} Hawking, S. W. and Ellis, G. F. R., {\it The large scale structure of space-time}, Cambridge Monographs
on Mathematical Physics, 1. Cambridge University Press, London–New York, 1973.

\bibitem{KSob} Klainerman, S. and Rodnianski, I.,
{\it A Kirchhoff-Sobolev parametrix for the wave equation and applications},
J. Hyperbolic Differ. Equ. 4 (2007), no. 3, 401--433.

\bibitem{KR2008} Klainerman, S. and Rodnianski, I.,
{\it  On the radius of injectivity of null hypersurfaces},
J. Amer. Math. Soc. 21 (2008), no. 3, 775--795.

\bibitem{KRpara} Klainerman, S. and Rodnianski, I., {\it Bilinear estimates on curved space-times},  J. Hyperbolic Differ. Equ. 2(2), 279–291 (2005)

\bibitem{KR2} Klainerman, S. and Rodnianski, I.,
{\it On the breakdown criterion in general relativity},  J. Amer. Math. Soc., 23 (2010), no. 2, 345--382.

\bibitem{Eric} Poisson, E., Pound A., and Vega, I.,
{\it The Motion of Point Particles in Curved Spacetime},
Living Rev. Relativity 14,  (2011), 7. http://www.livingreviews.org/lrr-2011-7

\bibitem{Shao} Shao, A.,
{\it Breakdown Criteria for Einstein Equations with Matters},
PhD thesis, Princeton University, 2010.

\bibitem{Sm} Smith, H. F.,
{\it A parametrix construction for wave equations with $C^{1,1}$ coefficients},
Ann. Inst. Fourier (Grenoble), 48 (1998), no. 3, 797--835.

\bibitem{Wangthesis} Wang, Q., {\it Causal Geometry of Einstein-Vacuum Spacetimes}, PhD thesis, Princeton University,  2006.

\bibitem{Wang10} Wang, Q.,
{\it Improved breakdown criterion for Einstein vacuum equations in CMC gauge},
Comm. Pure Appl. Math., Vol. LXV, 21--76 (2012).

\bibitem{WangCMC} Wang, Q.,
{\it Rough solutions of   Einstein vacuum equations in CMCSH gauges},
Comm. Math. Phys., 328 (2014), no. 3, 1275--1340.

\bibitem{Wang16} Wang, Q.,
{\it An intrinsic hyperboloid approach for Einstein Klein-Gordon equations},
preprint 2016, Arxiv:1607.01466
\end{thebibliography}
\end{document}